\documentclass[reqno]{amsart}

\usepackage[usenames,dvipsnames]{pstricks}
\usepackage{epsfig}
\allowdisplaybreaks
\usepackage{color}
\date{\today}

\usepackage[latin1]{inputenc}
\usepackage[T1]{fontenc}
\usepackage{amsmath, amsthm}
\usepackage[english]{babel}
\usepackage{amssymb}
\usepackage{latexsym}
\usepackage{graphicx}
\usepackage[all]{xy}

\newtheorem{theorem}{Theorem}[section]
\newtheorem{lemma}[theorem]{Lemma}
\newtheorem{proposition}[theorem]{Proposition}
\newtheorem{corollary}[theorem]{Corollary}
\theoremstyle{definition}
\newtheorem{definition}[theorem]{Definition}
\newtheorem{example}[theorem]{Example}

\theoremstyle{remark}
\newtheorem{remark}[theorem]{Remark}

\def\<{\langle}
\def\>{\rangle}
\def\a{\alpha}
\def\b{\beta}

\def\ci{\circ}
\def\c{\cdot}

\def\D{\Delta}

\def\g{\gamma}

\def\o{\otimes}

\def\v{\varepsilon}
\def\vp{\varphi}

\def\<{\langle}
\def\>{\rangle}
\def\m{\mu}

\begin{document}

\begin{center}

{\huge{\bf Pivotal and Ribbon Entwining Datums}}

\end{center}

\ \\
\begin{center}
{\bf Xiaohui ZHANG$^1$, Wei WANG$^{2}$, Xiaofan ZHAO$^3$
}
\end{center}

\ \\
\hspace{-0,5cm}$^{1}$ zxhhhhh@hotmail.com, School of Mathematical Sciences, Qufu Normal University, Qufu Shandong 273165, P. R. China.\\
\hspace{-0,5cm}$^{2}$ weiwang2012spring@yahoo.com, Department of Mathematics, Southeast University, Nanjing Jiangsu 210096, P. R. China.\\
\hspace{-0,5cm}$^{3}$ zhaoxiaofan8605@126.com, College of Mathematics and Information Science, Henan Normal University,
Xinxiang Henan 453007, P. R. China.\\

{\bf Abstract.} Let $(C,A,\varphi)$ be an entwining structure over $k$. In this paper, we introduce the notions of the pivotal entwined datums and ribbon entwined datums to generalize (co)pivotal Hopf algebras and (co)ribbon Hopf algebras. These notions give necessary and sufficient conditions for the category of entwined modules to be a pivotal category and ribbon category.
\\



\section{\textbf{Introduction}}
\setcounter{equation}{0}

Monoidal category theory played an important role in the theory
of knots and links and the theory of quantum groups. Through the reconstruction theory and Tannakian duality (\cite{ar}, \cite{ul}), quantum groups and monoidal categories are correspondence with each other.
There are many kinds of monoidal categories with additional structure - braided, rigid, pivotal, balanced, ribbon, etc., and many of them have an associated form in low dimensional topology theory and knot theory.
For example, ribbon category (see \cite{TRV}) is based on the isotopy invariants of
framed tangles;
spherical category (see \cite{BJW} and \cite{PD}) is based on the Turaev-Virostate sum model invariant of a closed piecewise-linear 3-manifold.
From the reconstruction theoretical point of view, a ribbon (resp. pivotal) category is equivalent to the category of (co)modules over a (co)ribbon (resp. pivotal) Hopf algebras (or its generalizations)(see \cite{jB}, \cite{ABAV} and \cite{ZZW}).

Pivotal category (or equivalently, a sovereign category) is introduced by Freyd and Yetter (see \cite{PD}, \cite{DY}) in 1992.
A pivotal structure in a rigid category is a monoidal natural transformation between the identity functor and the double (right) dual functor.
A pivotal Hopf algebra has by definition
a group-like element that implements the square of the antipode (called a
pivot) (see \cite{AAIT}).
Note that pivot could induce a pivotal structure in the representation category.
In 1995, Maltsiniotis (\cite{gm}) studied pivotal monoidal categories and proposed a new equivalent definition (which avoids the choice
of an autonomous structure), and showed that an axiom was redundant (see \cite{jB}).
Note that a theorem
of Deligne (Proposition 2.11 in \cite{DY}) showed that there is a sovereign structure on an autonomous braided category if and only if there is a twist (or a balanced structure) on it.
Thus a braided pivotal category is the same thing as a balanced autonomous category (see Corollary 4.21, \cite{Ps}).

Recently, pivotal structures have been studied by many scholars. In 2001, Bichon (\cite{jB}) introduced the (co)sovereign Hopf algebras and described the relation between the cosovereign Hopf algebras and sovereign categories. In 2004, Schauenburg (see \cite{SC}) introduced a categorical definition of degree 2 Frobenius-Schur indicators and gave
a different proof of the Frobenius-Schur Theorem for quasi-Hopf algebras based on the pivotal structures.
In 2012, Ng (see \cite{Ng}) researched a family of
pivotal finite-dimensional Hopf algebras with unique pivotal element.
In 2015, Kenichi Shimizu (see \cite{sk1,sk2}) use the theory of the pivotal cover to lay out a categorial approach to generalized
indicators for a non-semisimple Hopf algebra. In 2016, he also classified the pivotal structures of the Drinfeld center of
a finite tensor category (see \cite{sk3}).

It is interesting and important to find more new pivotal structures.
How to get more new examples of pivotal category? This is the first motivations of the present article. In order to investigate this question, we consider the necessary and sufficient condition for the category of entwined
modules to be a pivotal category.

Entwining structure is proposed by Brzezinski and Majid in \cite{ts} to define coalgebra principal bundles. An entwining structure over a monoidal category $\mathcal{C}$ consists of an algebra $A$, a coalgebra
$C$ and a morphism $\varphi:C\o A\rightarrow A \o C$ satisfying some axioms. The entwining modules are both $A$-modules and $C$-comodules, with compatibility relation given by $\varphi$. Note that the definition of entwined modules generalizes lots of important modules such as Hopf modules, Doi-Hopf modules, and Yetter-Drinfeld modules.

The second motivations of our paper raised from the study of how to get the ribbon structure in the center of the module category of a finite dimensional Hopf algebra.

A ribbon structure (see \cite{TRV}, and also see \cite{Kassel}) in a rigid braided category is a self-dual twist (or a self-dual balanced structure), which is a natural isomorphism from the identity functor to itself and compatible with the duality and the braiding.
In 1993, Kauffman and Radford got a necessary and sufficient condition for a finite-dimensional Drinfel'd double to be a ribbon Hopf algebra
(see \cite{ld}). As well-known, when a Hopf algebra $H$
is finite-dimensional we have that the category of modules of Drinfel'd double $D(H)$ is actually the Yetter-Drinfeld category $\mathbf{YD}_H^H$. Therefore one is prompted to ask the following question: when $\mathbf{YD}^H_H$ becomes a ribbon category from the point of view of the category theory?
In order to study the question, we discuss that how the ribbon structures appear in the category of entwined modules.

The paper is organized as follows. In Section 2 we recall some notions of entwining structures, pivotal categories and ribbon categories. In section 3, we describe the duality in the category of entwined modules $\mathcal{C}^C_A(\varphi)$ and show $\mathcal{C}^C_A(\varphi)$ is a rigid category. In section 4, we mainly give a necessary and sufficient condition for $\mathcal{C}^C_A(\varphi)$ to be a pivotal category.
In section 5, we give a necessary and sufficient condition for $\mathcal{C}^C_A(\varphi)$ to admit a ribbon structure.
In Section 6, we mainly discuss the Hopf algebras which are induced by $\mathcal{C}^C_A(\varphi)$, and show its (co)representation category is monoidal identified to $\mathcal{C}^C_A(\varphi)$.
Finally, we consider the Yetter-Drinfeld category and the category of generalized Long dimodules for application.

\section{\textbf{Preliminaries}}
\def\theequation{2.\arabic{equation}}
\setcounter{equation} {0} \hskip\parindent

Throughout the paper, we let $k$ be a fixed
field and $char(k)=0$ and $Vec_k$ be the category of finite dimensional $k$-spaces. All the algebras and coalgebras, modules and comodules are supposed to be in $Vec_k$. For the comultiplication
 $\D $ of a $k$-module $C$, we use the Sweedler-Heyneman's notation: $\Delta(c)=c_{1}\o c_{2}$,
for any $c\in C$. $\tau$ means the flip map $\tau(a \o b) = b \o a$.

\vskip 0.4cm
 {\bf 2.1. Entwining structure and entwined modules.}
\vskip 0.4cm

In this section, we will review several definitions related to entwined modules (see \cite{ts} or \cite{DHBP}).

Let $(C,\Delta_C,\varepsilon_C)$ be a coalgebra and $(A,m_A,\eta_A)$ an algebra over $k$. A map $\varphi: C \o A \rightarrow A \o C$, $\varphi(c \o a) = \sum a_\varphi \o c^\varphi$, is called an \emph{entwining map} if the following identities hold\\
$\left\{\begin{array}{l}
(E1)~\sum(ab)_\varphi \o c^\varphi = \sum a_\varphi b_\psi \o c^{\varphi\psi};\\
(E2)~\sum a_\varphi \o (c^\varphi)_1\o (c^\varphi)_2 = \sum a_{\varphi\psi} \o (c_1)^\psi \o (c_2)^\varphi;\\
(E3)~\sum (1_A)_\varphi \o c^\varphi = 1_A \o c;\\
(E4)~\sum a_\varphi \varepsilon_C(c^\varphi) = a\varepsilon_C(c),
\end{array}\right.$
\\
where $a,b \in A$, $c \in C$, $\psi=\varphi$.
Furthermore, $(C, A, \varphi)$ is called \emph{a right-right entwining structure}.

Let $\varphi : C \o A \rightarrow A \o C$ be an entwining map, $M\in \mathcal {C}$, $(M, \varrho_M)$ be a right $A$-module, $(M, \rho^M)$ be a right $C$-comodule. If the diagram
$$\aligned
\xymatrix{
M\o A \ar[d]_{\rho^M \o id_A} \ar[r]^{\varrho_M} & M \ar[r]^{\rho^M} & M\o C  \\
 M\o C \o A \ar"2,3"^{id_M\o\varphi} &  & M \o A\o C \ar[u]_{\varrho_M \o id_C}  }
\endaligned \eqno{(E0)}$$
is commutative, then we call the triple $(M, \varrho_M, \rho^M)$ an \emph{entwined module}.

The morphism between entwined modules is called \emph{entwined module morphism} if it is both $A$-linear and $C$-colinear. The category of entwined modules
 is denoted by $\mathcal {C}^C_A(\varphi)$.

Recall from \cite{BT2} and \cite{DHBP}, for any $g,f \in hom_k(C,A)$, one can define their \emph{entwined convolution product} $g \star f \in hom_k(C,A)$ via
\begin{eqnarray*}
(g \star f)(c):=\sum f(c_2)_\varphi g({c_1}^\varphi),~~~~c \in C.
\end{eqnarray*}
Note that the unit is $\eta_A \ci \varepsilon_C$.

Similarly, $hom_{k}(C\o C,A\o A)$ is also an algebra with the following \emph{entwined convolution product}:
$$\aligned
(g'\star f') :=& m_{A\o A}(A\o A\o g')(A\o \tau_{C,A}\o C)(\varphi\o\varphi)(C\o \tau_{C,A}\o A) (C\o C\o f')\Delta_{C\o C},
 \endaligned$$
where $g', f' \in hom_{k}(C\o C,A\o A)$. Note that the unit is $(\eta_A\o\eta_A) \ci (\varepsilon_C\o \v_C)$.

Recall from [\cite{BT}, Corollary 3.4,] (or see \cite{DHBP}) that $C\o A$ is an object in $\mathcal{C}_A^C(\varphi)$ via
\begin{eqnarray*}
&(c \o a)\c x = c \o ax;\\
&(c \o a)_0 \o (c \o a)_1= \sum c_1 \o a_\varphi \o {c_2}^\varphi,
\end{eqnarray*}
where $a,x \in A$ and $c \in C$.

$A\o C$ is also an object in $\mathcal{C}_A^C(\varphi)$ via
\begin{eqnarray*}
&(a \o c)\c x = \sum ax_\vp \o c^\vp;\\
&(a \o c)_0 \o (a \o c)_1= a \o c_1 \o c_2.
\end{eqnarray*}

Further, for any right $A$-module $M$, $M \o C$ is also an entwined module by
\begin{eqnarray*}
&(m \o c)\c a = \sum m \c a_\varphi \o c^\varphi;\\
&(m \o c)_0 \o (m \o c)_1= m \o c_1 \o c_2.
\end{eqnarray*}
This defines a right adjoint functor for the underlying functor
$U: \mathcal{C}_A^C(\varphi)\rightarrow \mathcal{M}_A$.

\vskip 0.4cm
 {\bf 2.2. Monoidal entwining datum and double quantum group.}
\vskip 0.4cm

Suppose that $C$ and $A$ are two bialgebras over $k$ such that $(C, A, \varphi)$ is an entwining structure. Recall that $(C, A, \varphi)$ is called a \emph{monoidal entwining datum} if the following equations hold\\
$\left\{\begin{array}{l}
(E5)~\sum (a_\varphi)_1 \o (a_\varphi)_2 \o (cd)^\varphi = \sum (a_1)_{\varphi} \o (a_2)_\psi \o c^\varphi d^\psi;\\
(E6)~\sum\varepsilon_A(a_\varphi)(1_C)^\varphi = \varepsilon_A(a)1_C,
\end{array}\right.$
\\
where $a \in A$, $c \in C$.

Recall from [\cite{DHBP}, Theorem 4.1] that
$\mathcal {C}^C_A(\varphi)$ is a monoidal category such that the forgetful functors are strict monoidal if and only if $(C, A, \varphi)$ is a monoidal entwining datum. Further, for any $M,N \in \mathcal {C}^C_A(\varphi)$, the $A$-action and the $C$-coaction on $M \o N$ are given by
\begin{eqnarray*}
&(m \o n)\c a = m\c a_1 \o n\c a_2;\\
&(m \o n)_0 \o (m \o n)_1= m_0 \o n_0 \o m_1n_1,
\end{eqnarray*}
where $m \in M$, $n \in N$, $a \in A$. Moreover, the tensor unit in $\mathcal{C}_A^C(\varphi)$ is $(k,id_k \o \varepsilon_A, id_k \o \eta_C)$.

Recall that
a pair of bialgebras $C$ and $A$ together with a
monoidal entwining map $\vp$ (such that $\mathcal {C}^C_A(\varphi)$ is a
monoidal category) and together with a $k$-linear morphism $R:C\o C\rightarrow A \o A$ is
called a \emph{double quantum group} if the following identities hold for any $a \in A$, $c,d,x,y,z \in C$:\\
$\left\{\begin{array}{l}
(E7)~\sum R(c_1 \o d_1) \o c_2d_2 = {R^{(1)}(c_2 \o d_2)}_\vp \o {R^{(2)}(c_2 \o d_2)}_\psi \o {d_1}^\vp {c_1}^\psi; \\
(E8)~\sum {a_2}_\psi R^{(1)}(c^\vp \o d^\psi) \o {a_1}_\vp R^{(2)}(c^\vp \o d^\psi) = R^{(1)}(c \o d) a_1 \o R^{(2)}(c \o d) a_2; \\
(E9)~\sum {R^{(1)}(x \o yz)}_1 \o {R^{(1)}(x \o yz)}_2 \o R^{(2)}(x \o yz) \\
~~~~~~~~~~~~~~~~~~~= r^{(1)}(x_2 \o y) \o R^{(1)}({x_1}^\vp \o z) \o {r^{(2)}(x_2 \o y)}_\vp R^{(2)}({x_1}^\vp \o z); \\
(E10)~\sum R^{(1)}(xy \o z) \o {R^{(2)}(xy \o z)}_1\o {R^{(2)}(xy \o z)}_2 \\
~~~~~~~~~~~~~~~~~~~= {r^{(1)}(y \o z_2)}_\vp R^{(1)}(x \o {z_1}^\vp) \o R^{(2)}(x \o {z_1}^\vp) \o r^{(2)}(y \o z_2); \\
(E10)~R\in hom_k(C \o C, A\o A)\mbox{~is~invertible~under~the~entwined~convolution}.
\end{array}\right.$
\\

Recall from [\cite{DHBP}, Theorem 5.5] that
$\mathcal {C}^C_A(\varphi)$ is a braided monoidal category if and only if $(C, A, \varphi, R)$ is a double quantum group.
Further, the braiding $\mathbf{C}$ in $\mathcal {C}^C_A(\varphi)$ is given by
\begin{eqnarray}
\mathbf{C}_{M,N}: M \o N\rightarrow N \o M,~~~~m \o n \mapsto (n_0 \o m_0)\c R(m_1 \o n_1),
\end{eqnarray}
where $M, N\in  \mathcal {C}^C_A(\varphi)$.

\vskip 0.4cm
 {\bf 2.3. Pivotal category and ribbon category.}
\vskip 0.4cm

In this section, we will review several definitions and notations related to ribbon structures.

Let $(\mathcal {C}, \o, I)$ and $(\mathcal {C}', \o', I')$ be strict monoidal categories. A functor $G:\mathcal {C}\rightarrow \mathcal {C}'$ is called a \emph{monoidal functor}, if there exists a natural transformation $G_2: G\otimes' G \rightarrow G\otimes$ and a morphism $G_0: I' \rightarrow G(I)$ in $\mathcal {C}'$, such that for any $X, Y, Z \in \mathcal {C}$, the following diagrams commute
$$\aligned
\xymatrix{
G(X) \o' G(Y) \o' G(Z) \ar[d]_{G_2(X,Y)\otimes' id_{G(Z)}} \ar[r]^{~~~id_{G(X)} \otimes' G_2(Y, Z)} & G(X) \o' G(Y\o Z) \ar[d]^{G_2(X,Y\otimes Z)} \\
G(X \o Y) \o' G(Z) \ar[r]_{~~G_2(X\otimes Y, Z)} & G(X \o Y \o Z),}~~~~
\xymatrix{
G(X) \ar[r]^{id_{G(X)}\otimes' G_0~~} \ar[dr]|{id_{G(X)}}\ar[d]_{G_0 \otimes' id_{G(X)}} & G(X)\o' I'\ar[d]^{G_2(X,I)} \\
I' \o' G(X) \ar[r]_{~~~G_2(I,X)} & G(X) . }
 \endaligned$$

A monoidal functor $(G,G_2,G_0)$ is said to be \emph{strong} (respectively \emph{strict}) if $G_2$ and $G_0$
are isomorphisms (respectively identities).

Let $(\mathcal{C}, \o,I)$ and $(\mathcal{C}', \o',I')$ be monoidal categories, $F,G$ be monoidal functors from $\mathcal{C}$ to $\mathcal{C}'$. A natural transformation $\alpha:F \Rightarrow G : \mathcal{C}\rightarrow \mathcal{C}'$ is called \emph{monoidal} if $\a$ satisfies
$$
\alpha_{X \o Y}\circ F_2(X,Y) = G_2(X,Y)\circ (\alpha_X \o' \alpha_Y),~~\mbox{and}~~G_0 = \alpha_IF_0.
$$
Or equivalently (\cite{Etingof}, Definition 1.5.1),
if $\a$ satisfies
$$
\alpha_{X \o Y}\circ F_2(X,Y) = G_2(X,Y)\circ (\alpha_X \o'\alpha_Y), ~~\mbox{and}~~\alpha_I~\mbox{is~an~isomorphism}.
$$

Let $(\mathcal {C}, \o, I)$ be a strict monoidal category. Recall from \cite{Kassel} or \cite{jB} that for an object $V \in \mathcal {C}$, a \emph{left dual} of $V$ is a triple $(V^\ast, ev_V, coev_V)$, where $V^\ast$ is an object, $ev_V : V^\ast \o V\rightarrow I$ and $coev_V : I\rightarrow V \o V^\ast$ are morphisms in $\mathcal {C}$, satisfying
$$ (V \o ev_V)(coev_V \o V) = id_V,\ \ \mbox{and} \  \ (ev_V \o V^\ast)(V^\ast \o coev_V) = id_{V^\ast}.$$

Similarly, a \emph{right dual} of $V$ is a triple $({}^\ast V, \widetilde{ev}_V, \widetilde{coev}_V)$, where ${}^\ast V$ is an object, $\widetilde{ev}_V : V \o {}^\ast V\rightarrow I$ and $\widetilde{coev}_V : I\rightarrow {}^\ast V \o V$ are morphisms in $\mathcal {C}$, satisfying
$$ (\widetilde{ev}_V\o V )(V \o \widetilde{coev}_V ) = id_V,\ \ \mbox{and} \  \ ({}^\ast V \o \widetilde{ev}_V )(\widetilde{coev}_V \o {}^\ast V) = id_{{}^\ast V}.$$

If each object in $\mathcal {C}$ admits a left dual (respectively a right dual, respectively both a left dual and a right dual), then $\mathcal {C}$ is called a \emph{left rigid category} (respectively a \emph{right rigid category}, respectively a \emph{rigid category}).

Assume that $\mathcal {C}$ is a left rigid category. $X, Y \in \mathcal {C}$, for a morphism $g : Y \rightarrow X$ define its \emph{transpose} as follows:
$$g^\ast:=
\xymatrix{
   X^\ast \ar[rr]^-{id_{X^\ast} \o coev_Y} & &   X^\ast\o Y\o Y^\ast \ar[rr]^-{id_{X^\ast} \o g \o id_{Y^\ast}}& &  X^\ast\o X\o Y^\ast \ar[rr]^-{ev_X \o id_{Y^\ast}} &  & Y^\ast
   }. \eqno(TR1)$$
Then it is easy to get the following commutative diagrams
$$\aligned
\xymatrix{
X^\ast \o Y \ar[d]_{id_{X^\ast} \o g} \ar[r]^{g^\ast\o id_Y} & Y^\ast \o Y \ar[d]^{ev_Y} \\
X^\ast \o X \ar[r]_-{ev_X} & I,
}\ \ \ \ \
\xymatrix{
I \ar[d]_{coev_X } \ar[r]^-{coev_Y} & Y \o Y^\ast \ar[d]^{g\o id_Y^\ast} \\
X \o X^\ast \ar[r]_{id_{X} \o g^\ast } & X \o Y^\ast.
}
\endaligned \eqno(TR2)$$
Further, this defines a bijection between $Hom_{\mathcal {C}}(X^\ast, Y^\ast)$ and $Hom_{\mathcal {C}}(Y, X)$.

\begin{lemma}\label{2.1}
Let $\mathcal {C}$ be a left rigid category, $U, V, W$ be objects in $\mathcal {C}$. Then \\
(1). $V^\ast\o U^\ast\cong (U\o V)^\ast$;\\
(2). if $f:V\rightarrow W$ and $g:U\rightarrow V$ are morphisms in $\mathcal {C}$, then we have $(f\circ g)^\ast = g^\ast \circ f^\ast$, and $(1_V)^\ast = 1_{V^\ast}$;\\
(3). $I^\ast = I$.
\end{lemma}

Thus we have a strong monoidal functor $()^\ast : \mathcal {C}^{op}\rightarrow  \mathcal {C}$, where $\mathcal {C}^{op}$ is the opposite category to $\mathcal {C}$ with opposite monoidal structure. $()^\ast$ is called the \emph{left dual functor}.

Similarly, for a morphism $f : X\rightarrow Y$ in a right rigid category $\mathcal {C}$, define ${}^\ast f: {}^\ast Y\rightarrow {}^\ast X$ by
$$ {}^\ast f = ({}^\ast X \o \widetilde{ev}_Y)({}^\ast X \o f \o {}^\ast Y )(\widetilde{coev}_X\o {}^\ast Y),$$
then we get the \emph{right dual functor} ${}^\ast () : \mathcal {C}^{op}\rightarrow \mathcal {C}$. It is also a strong monoidal functor. Further, if $\mathcal {C}=Vec_k$, then $()^\ast$ and ${}^\ast ()$ are all strict monoidal functors.

Recall from \cite{Etingof} and \cite{Ps} that a \emph{pivotal structure} in a right rigid category $\mathcal {C}$ is an isomorphism $\beta:id\rightarrow {}^{\ast\ast}()$ of monoidal functors, where $id$ means the identity functor and ${}^{\ast\ast}() = {}^{\ast}() \ci {}^{\ast}()$. A \emph{pivotal category} is a right rigid category endowed with a pivotal structure.

Let $(\mathcal{C}, \o, I, a,l,r,\mathbf{C})$ be a braided monoidal category with the braiding $\mathbf{C}$. Recall from \cite{Kassel} (or \cite{Ps}) that a \emph{twist} (or a \emph{balanced structure})
on $\mathcal{C}$ is a family $\theta_V : V \rightarrow V$ of natural isomorphisms
indexed by the objects $V$ of $\mathcal{C}$ satisfying
\begin{eqnarray*}
\theta_{V \o W} = \mathbf{C}_{W,V} \mathbf{C}_{V,W}(\theta_V \o \theta_W).
\end{eqnarray*}

A twist $\theta$ on an autonomous category $\mathcal{C}$ is \emph{self-dual} if $\theta_{V^\ast} = (\theta_V)^\ast$ (or equivalently,
$\theta_{{}^\ast V} = {}^\ast(\theta_V)$).

A \emph{ribbon category} is a braided autonomous category endowed with a self-dual twist.

\section{\textbf{The duality in the category of entwined modules}}
\def\theequation{3.\arabic{equation}}
\setcounter{equation} {0} \hskip\parindent

From now on we assume that $C$ and $A$ are two Hopf algebras with bijective antipodes over $k$, and $\varphi:C\o A \rightarrow A\o C$ is a $k$-linear map such that $(C,A,\varphi)$ is a monoidal entwining datum.

For any $(M, \varrho_M, \rho^M) \in \mathcal {C}^C_A(\varphi)$, set $M^\ast = {}^\ast M = hom_k(M,k)$ as spaces, and define the $A$-action and $C$-coaction on $M^\ast$ and ${}^\ast M$ by
\begin{equation*}
\begin{split}
\varrho_{M^\ast}&:  M^\ast \o A\longrightarrow M^\ast,~~
 (f\cdot a)(m):=f(m\cdot S^{-1}_A(a)),\\
\rho^{M^\ast}&:  M^\ast \longrightarrow M^\ast \o C, ~~
 f_0(m) \o f_1:= f(m_{0}) \o S_C(m_{1}), \\
\varrho_{{}^\ast M}&:  {}^\ast M \o A\longrightarrow {}^\ast M,~~
 (f\cdot a)(m):=f(m\cdot S_A(h)),\\
\rho^{{}^\ast M}&: {}^\ast M  \longrightarrow {}^\ast M \o C, ~~
 f_0(m) \o f_1:= f(m_{0}) \o S^{-1}_C(m_{1}),
\end{split}
\end{equation*}
where $f \in hom_k(M,k)$, $a \in A$, $m \in M$, and define the evaluation map and coevaluation map by
\begin{eqnarray*}
ev_M:   f \otimes  m \longmapsto  f(m); ~~~~
coev_M:  1_k \longmapsto  \sum_i e_i \o e^i; \\
\widetilde{ev}_M:   m \otimes  f \longmapsto  f(m); ~~~~
\widetilde{coev}_M:  1_k \longmapsto  \sum_i e^i \o e_i,
\end{eqnarray*}
where $e_i$ and $e^i$ are dual bases in $M$ and $M^\ast$.
It is easy to check that $M^\ast$ and ${}^\ast M$ are all both $A$-modules and $C$-comodules. Further, $ev,coev,\widetilde{ev},\widetilde{coev}$ are all both $A$-linear and $C$-colinear maps.

\begin{lemma} \label{3.0}
For any $a \in A$, $c \in C$, the following identities hold
\begin{eqnarray}
&S_A^{-1}(a) \o S_C(c) = \sum (S_A^{-1}(a_\varphi))_\psi \o (S_C(c^\psi))^\varphi;\\
&S_A(a) \o S^{-1}_C(c) = \sum (S_A(a_\varphi))_\psi \o (S^{-1}_C(c^\psi))^\varphi.
\end{eqnarray}
\end{lemma}

\begin{proof}
See Section 6.
\end{proof}

\begin{theorem}\label{3.1}
$\mathcal {C}^C_A(\varphi)$ is a rigid category.
\end{theorem}

\begin{proof}
We only show that $\mathcal {C}^C_A(\varphi)$ admit left duality, i.e.,
$M^\ast$ is an object in $\mathcal {C}^C_A(\varphi)$ for any entwined module $M$. Actually, for any
$\m \in M^\ast$, $m \in M$, $a \in A$, we have
\begin{eqnarray*}
(\m\c a)_0(m) \o (\m\c a)_1&=& \m(m_0 \c S^{-1}_A(a)) \o S_C(m_1)\\
&\stackrel {(3.1)}{=}&\sum \m(m_0 \c (S_A^{-1}(a_\varphi))_\psi) \o (S_C({m_1}^\psi))^\varphi\\
&=&\sum \m_0(m\c S_A^{-1}(a_\varphi)) \o {\m_1}^\varphi=\sum(\m_0 \c a_\varphi)(m) \o (\m_1)^\varphi.
\end{eqnarray*}
Thus $\mathcal {C}^C_A(\varphi)$ is a left rigid category.

Similarly, one can check that $\mathcal {C}^C_A(\varphi)$ is a right rigid category by using Eq.(3.2).
\end{proof}

\section{\textbf{The pivotal structure in the category of entwined modules}}
\def\theequation{4.\arabic{equation}}
\setcounter{equation} {0} \hskip\parindent

Suppose that $F$ is the underlying functor from $\mathcal {C}^C_A(\varphi)$ to $Vec_k$, $Nat(F \ci id, F \ci {}^{\ast\ast}())$ means the collection of natural transformations from the functor $F \ci id:{\mathcal {C}^C_A(\varphi)}\rightarrow Vec_k$ to the functor $F \ci {}^{\ast\ast}()$. Then we have the following property.

\begin{proposition}\label{4.1}
There is a bijective map between $Nat(F \ci id, F \ci {}^{\ast\ast}())$ and $hom_k(C,A)$.
\end{proposition}

\begin{proof}
Define a map $P:Nat(F \ci id, F \ci {}^{\ast\ast}())\rightarrow hom_k(C,A)$ by
\begin{eqnarray*}
P(\beta):C\rightarrow A, ~~~~c\mapsto\sum \beta_{A \o C}(1_A \o c)(\varepsilon_C \o e^i)e_i,
\end{eqnarray*}
where $\beta \in Nat(F \ci id, F \ci {}^{\ast\ast}())$, $c \in C$, $e_i$ and $e^i$ are bases of $A$ and $A^\ast$ respectively, dual to each other.
Define $Q:hom_k(C,A)\rightarrow Nat(F \ci id, F \ci {}^{\ast\ast}())$ by
$$
Q(\mathfrak{g})_M: M\rightarrow {}^{\ast\ast} M,~~m\mapsto Q(\mathfrak{g})_M(m),~~\mbox{such~that}~~Q(\mathfrak{g})_M(m)(\mu)=\mu(m_0\c \mathfrak{g}(m_1)),
$$
where $\mathfrak{g} \in hom_k(C,A)$, $M \in \mathcal {C}^C_A(\varphi)$, $\mu \in {}^\ast M$, $m \in M$.

Firstly, for any morphism $f:N\rightarrow M$ in $\mathcal {C}^C_A(\varphi)$, $\nu \in {}^\ast N$, $n \in N$, since
\begin{eqnarray*}
{}^{\ast\ast}f(Q(\mathfrak{g})_M(m))(\nu)&=&Q(\mathfrak{g})_M(m)({}^{\ast}f(\nu))\\
&=&{}^{\ast}f(\nu)(m_0 \c \mathfrak{g}(m_1))=\nu ({f(m)}_0 \c \mathfrak{g}({f(m)}_1)),
\end{eqnarray*}
the diagram
$$\aligned
\xymatrix{
M \ar[d]_{f}\ar[rr]^{Q(\mathfrak{g})_M} && {}^{\ast\ast} M \ar[d]^{{}^{\ast\ast} f} \\
N \ar[rr]^{Q(\mathfrak{g})_N} && {}^{\ast\ast} N
}
\endaligned$$
is commute. Thus $Q$ is well-defined.

Secondly, for $g \in hom_k(C,A)$, $c \in C$, we compute
\begin{eqnarray*}
P(Q(\mathfrak{g}))(c)&=&\sum Q(\mathfrak{g})_{A \o C}(1_A \o c)(\varepsilon_C \o e^i)e_i\\
&=&\sum (\varepsilon_C \o e^i)((1_A \o c)_0\c \mathfrak{g}((1_A \o c)_1))e_i\\%
&=&\sum \varepsilon_C(c_1^\varphi)e^i(\mathfrak{g}(c_2)_\varphi)e_i=\mathfrak{g}(c).
\end{eqnarray*}

Thirdly, for $\sigma \in Nat(F \ci id, F \ci {}^{\ast\ast}())$, $\mu \in {}^\ast M$, $m \in M$, notice that
\begin{eqnarray*}
Q(P(\beta))_M(m)(\m)&=&\m(m_0\c P(\beta)(m_1))\\
&=&\sum \m(m_0 \c e_i) \beta_{A \o C}(1_A \o m_1)(\varepsilon_C \o e^i).
\end{eqnarray*}
For any right $A$-module $U$, $u \in U$, one can define the following $A$-linear map
$$
\widetilde{u}:A\rightarrow U, ~~~~a\mapsto u\c a,
$$
it is easy to check that $\widetilde{u} \o id_C: A \o C\rightarrow U \o C$ is both $A$-linear and $C$-colinear. Actually, for any $a,x \in A$, $c \in C$, we compute
$$\aligned
\xymatrix{
a \o c \o x \ar[d]_{\theta_{A \o C}}\ar[rr]^{\widetilde{u} \o id_C \o id_A} && u\c a \o c \o x \ar[d]^{\theta_{U \o C}} \\
ax_\varphi \o c^\varphi \ar[rr]^{\widetilde{u} \o id_A} && u\c ax_\varphi \o c^\varphi,
}~~~~
\xymatrix{
a \o c \ar[d]_{\widetilde{u} \o id_C}\ar[rr]^{\rho^{A \o C}} && a \o c_1 \o c_2 \ar[d]^{\widetilde{u} \o id_C\o id_C} \\
u \c a \o c \ar[rr]^{\rho^{U \o C}} && u\c a \o c_1 \o c_2.
}
\endaligned$$
Thus we have
\begin{eqnarray*}
&&~~~\sum \m(m_0 \c e_i) \beta_{A \o C}(1_A \o m_1)(\varepsilon_C \o e^i)\\
&&=\sum \beta_{A \o C}(1_A \o m_1)(\varepsilon_C \o (\m \ci \widetilde{m_0})(e_i)e^i)\\
&&=(({}^{\ast\ast} \widetilde{m_0} \o id_{{}^{\ast\ast} C}) \ci \beta_{A \o C})(1_A \o m_1)(\varepsilon_C \o \m)\\
&&=\beta_{M \o C}(m_0 \o m_1)(\varepsilon_C \o \m)\\
&&=(\beta_{M \o C} \ci \rho^M )(m) (\varepsilon_C \o \m).
\end{eqnarray*}
Then from the following commute diagram
$$\aligned
\xymatrix{
{}^\ast C \o {}^\ast M\o M \ar[d]_{id_{{}^\ast C } \o id_{{}^\ast M} \o \rho^M} \ar[rr]^{id_{{}^\ast C } \o id_{{}^\ast M} \o \b_M} & &  {}^\ast C \o {}^\ast M \o {}^{\ast\ast}M \ar[d]|{id_{{}^\ast C } \o id_{{}^\ast M} \o {}^{\ast\ast}\rho^M }\ar[rr]^{{}^\ast(\rho^M) \o id_{{}^{\ast\ast} M}} & & {}^\ast M\o {}^{\ast\ast}M \ar[d]^{\widetilde{ev}_{{}^\ast M}} \\
{}^\ast C \o {}^\ast M\o M \o C \ar[rr]_{id_{{}^\ast C } \o id_{{}^\ast M} \o \beta_{M\o C}} & & {}^\ast C \o {}^\ast M\o {}^{\ast\ast} M \o {}^{\ast\ast} C \ar[rr]_{~~~~~~~~~~\widetilde{ev}_{{}^\ast C \o {}^\ast M}} & & k,
}
\endaligned$$
we have
\begin{eqnarray*}
(\beta_{M \o C} \ci \rho^M )(m) (\varepsilon_C \o \m)&=&\beta_{M}(m)({}^\ast(\rho^M) (\varepsilon_C \o \m) )\\
&=&\beta_M(m)(\m).
\end{eqnarray*}
This completes the proof.
\end{proof}

From now on, assume that $\beta \in Nat(F \ci id, F \ci {}^{\ast\ast}())$ and $\mathfrak{g} \in hom_k(C,A)$ are in correspondence with each other.

\begin{lemma} \label{4.2}
$\beta$ is a monoidal natural transformation if and only if $g$ satisfies
\begin{eqnarray}
\Delta_A(\mathfrak{g}(cd))=\mathfrak{g}(c) \o \mathfrak{g}(d),
\end{eqnarray}
for any $c,d \in C$, and
\begin{eqnarray}
\varepsilon_A(\mathfrak{g}(1_C))=1_k.
\end{eqnarray}
\end{lemma}

\begin{proof}
Firstly, since $(k,id_k \o \varepsilon_A, id_k \o \eta_C) \in \mathcal{C}_A^C(\varphi)$, it is easy to check that $\sigma_k = id_k$ if and only if the following diagram
$$\aligned
\xymatrix{
k \ar[rrrrd]_{id_k} \ar[rr]^{id_k \o \eta_C} && k \o C \ar[rr]^{id_k \o \g} && k \o A \ar[d]^{id_k \o \varepsilon_A}\\
&&&&k
}
\endaligned$$
is commutative, i.e., the identity (4.2) holds.

Secondly, if the equation $(4.1)$ holds, then for any entwined modules $M, N$, $m \in M$, $n \in N$, $\mu \in {}^\ast M$, $\nu \in {}^\ast N$, we have
\begin{eqnarray*}
\sigma_{M \o N}(\nu \o \mu)(m \o n)&=&(\nu \o \mu)((m_0 \o n_0)\c \mathfrak{g}(m_1n_1))\\
&\stackrel {(4.1)}{=}&(\nu \o \mu)(m_0\c \mathfrak{g}(m_1) \o n_0\c \mathfrak{g}(m_2))\\
&=&(\sigma_{M} \o \sigma_{N})(\nu \o \mu)(m \o n),
\end{eqnarray*}
it follows that $\sigma$ is monoidal.

Conversely, if $\sigma$ is monoidal, we take $M=N=C \o A$. Then for any $c, d \in C$, $\delta,\gamma \in {}^\ast A$, for one thing,
we have
\begin{eqnarray*}
&&~~~\sigma_{(C \o A) \o (C \o A)}(\delta \o \varepsilon_C \o \gamma\o\varepsilon_C)(c\o 1_A \o d \o 1_A)\\
&&=(\delta \o \varepsilon_C \o \gamma\o\varepsilon_C)((c_1 \o 1_A)\c (\mathfrak{g}(c_2d_2))_1 \o (d_1 \o 1_A)\c (\mathfrak{g}(c_2d_2)_2))\\
&&=\varepsilon_C(c_1)\delta((\mathfrak{g}(c_2d_2))_1)\varepsilon_C(c_2)\gamma((\mathfrak{g}(c_2d_2))_2)\\
&&=(\delta \o \gamma)(\Delta_A(\mathfrak{g}(cd))).
\end{eqnarray*}
For another, we have
\begin{eqnarray*}
&&~~~(\sigma_{C \o A} \o \sigma_{C \o A})((\delta \o \varepsilon_C) \o (\gamma \o\varepsilon_C))((c\o 1_A) \o (d \o 1_A))\\
&&=(\delta \o\varepsilon_C)((c_1 \o 1_A)\c \mathfrak{g}(c_2)) (\gamma \o\varepsilon_C)((d_1 \o 1_A)\c \mathfrak{g}(d_2)\\
&&=(\delta \o \gamma)(\mathfrak{g}(c) \o \mathfrak{g}(d)).
\end{eqnarray*}
Thus that Eq.(4.1) holds because of the arbitrariness of $\delta$ and $\gamma$.
\end{proof}

Recall from [\cite{ABAV}, Lemma 3.4], since $\mathcal{C}^C_A(\varphi)$ is a rigid category, we immediately get the following property.

\begin{proposition} \label{4.0}
If $\beta$ is a monoidal natural transformation, then $\mathfrak{g}$ is convolution invertible.
\end{proposition}

\begin{proof}
Define $P':Nat(F \ci {}^{\ast\ast}(), F \ci id)\rightarrow hom_k(C,A)$ by
\begin{eqnarray*}
P'(\beta'):C\rightarrow A, ~~~~c\mapsto \sum (id_A \o \varepsilon_C) (\b'_{A \o C}(1_{{}^{\ast\ast} A} \o f^i ) f_i(c),
\end{eqnarray*}
where $\beta' \in Nat(F \ci id, F \ci {}^{\ast\ast}())$, $1_{{}^{\ast\ast} A}$ means the unit in ${}^{\ast\ast} A$, $c \in C$, $f_i$ and $f^i$ are bases of ${}^{\ast} C$ and ${}^{\ast\ast} C$ respectively, dual to each other.

Define $Q':hom_k(C,A)\rightarrow Nat(F \ci {}^{\ast\ast}(), F\ci id)$ by
$$
Q'(h)_M: {}^{\ast\ast} M\rightarrow M,~~\varpi\mapsto \sum \varpi(\epsilon^i){\epsilon_i}_0 \c h({\epsilon_i}_1),
$$
where $h \in hom_k(C,A)$, $M \in \mathcal {C}^C_A(\varphi)$, $\varpi \in {}^{\ast\ast} M$, $\epsilon_i$ and $\epsilon^i$ are bases of $M$ and $M^\ast$ respectively, dual to each other.

Obviously $Q'$ is well-defined, and is the inverse of $P'$.

If $\beta$ is a monoidal natural transformation, then $\beta$ is invertible because [\cite{ABAV}, Lemma 3.4]. Suppose that $\b'$ is the inverse of $\b$. One can easily show that $P'(\b')=\mathfrak{g}'$ is the entwined convolution inverse of $P(\b)=\mathfrak{g}$.
\end{proof}

\begin{lemma}\label{4.4}
$\beta$ is $A$-linear if and only if for any $a \in A$, $c \in C$, $\mathfrak{g}$ satisfies
\begin{eqnarray}
\mathfrak{g}(c)S^2_A(a) = \sum a_\varphi \mathfrak{g}(c^\varphi).
\end{eqnarray}
\end{lemma}

\begin{proof}
$\Rightarrow$: Suppose that $\gamma \in {}^\ast A$, $c \in C$, $a \in A$. For one thing, we compute
\begin{eqnarray*}
(\b_{C \o A}(c \o 1_A) \c a)(\g \o \varepsilon_C)&=&\b_{C \o A}(c \o 1_A)((\g \o \varepsilon_C) \c S_A(a))\\
&=&(\g \o \varepsilon_C)((c \o  1_A)_0  \c \mathfrak{g}((c \o  1_A)_1) \c S^2_A(a))\\
&=&(\g \o \varepsilon_C)(c_1 \o \mathfrak{g}(c_2)S^2_A(a))=\g(\mathfrak{g}(c)S^2_A(a) ).
\end{eqnarray*}

For another, we have
\begin{eqnarray*}
(\beta_{C \o A}((c \o 1_A) \c a))(\g \o \varepsilon_C)&=&(\g \o \varepsilon_C) ((c \o a)_0 \c \mathfrak{g}((c \o a)_1))\\
&=&\sum (\g \o \varepsilon_C)(c_1 \o a_\varphi \mathfrak{g}({c_2}^\varphi))\\
&=&\sum \g(a_\varphi \mathfrak{g}(c^\varphi)).
\end{eqnarray*}

Since $\beta$ is $A$-linear, we have $\g(\mathfrak{g}(c)S^2_A(a))=\g(a_\varphi \mathfrak{g}(c^\varphi))$ which implies Eq.(4.3) holds.

$\Leftarrow$: For any $M \in \mathcal{C}_A^C(\varphi)$, $m \in M$, $\mu \in {}^\ast M$, $a \in A$, we have
\begin{eqnarray*}
\beta_{M}(m\c a)(\mu)&=&\sum \mu(m_0\c a_\varphi \mathfrak{g}({m_1}^\varphi))\\
&\stackrel{(4.3)}{=}&\mu(m_0\c \mathfrak{g}(m_1)S^2_A(a))\\
&=&\beta_{M}(m)(\m \c S_A(a)) = (\beta_{M}(m) \c a)(\m),
\end{eqnarray*}
it follows that $\beta$ is $A$-linear.
\end{proof}

\begin{lemma} \label{4.5}
$\beta$ is $C$-colinear if and only if $g$ satisfies
\begin{eqnarray}
\mathfrak{g}(c_1) \o c_2 = \sum {\mathfrak{g}(c_2)}_\varphi \o S_C^{-2}({c_1}^\varphi),
\end{eqnarray}
for any $c \in C$.
\end{lemma}

\begin{proof}
It can be proved by similar calculations in Lemma \ref{4.4}.
\end{proof}

\begin{definition}
Assume that $C, A$ are two Hopf algebras with bijective antipodes over a field $k$, and $(C, A, \varphi)$ is a monoidal entwining datum. If there exists a $k$-linear map $\mathfrak{g}:C\rightarrow A$, such that Eqs.(4.1)-(4.4) are satisfied. Then $\mathfrak{g}$ is called an \emph{entwined pivotal morphism} over $(C,A,\vp)$. Further, $(C,A,\varphi,\mathfrak{g})$ is called a \emph{pivotal entwined datum}.
\end{definition}

Combining Proposition \ref{4.1} - Lemma \ref{4.5}, we can get the following theorem.

\begin{theorem}\label{4.6}
Assume that $C,A$ are two Hopf algebras with bijective antipodes over $k$, $\varphi:C\o A \rightarrow A\o C$ is a $k$-linear map such that $(C, A, \varphi)$ is a monoidal entwining datum. Then $\mathcal {C}^C_A(\varphi)$ is a pivotal category with the pivotal structure $\beta$ if and only if there is an entwined pivotal morphism $\mathfrak{g}\in Hom_k(C,A)$. Moreover, $\beta$ is defined by
\begin{eqnarray*}
\beta_M:M\rightarrow {}^{\ast\ast} M,~~\b_M(m)(\m)=\m(m_0 \c \mathfrak{g}(m_1)),~~\mbox{where~}\m \in {}^\ast M,~~m \in M
\end{eqnarray*}
for any $(M,\theta_M,\rho^M) \in \mathcal {C}^C_A(\varphi)$.
\end{theorem}

\begin{example}
If $C=k$, $\varphi=id_A$, then the entwined pivotal morphism is an element $\mathfrak{g} \in A$ satisfies\\
$\left\{\begin{array}{l}
(1)~\Delta_A(\mathfrak{g}) = \mathfrak{g} \o \mathfrak{g};\\
(2)~\varepsilon_A(\mathfrak{g})=1_k;\\
(3)~\mathfrak{g}S_A(a)=S^{-1}_A(a)\mathfrak{g},
\end{array}\right.$
\\
for any $a \in A$, which implies $\mathfrak{g}$ is a \emph{pivot} (or a \emph{sovereign element}) in $A$ (see \cite{AAIT}, or \cite{jB}), and $A$ is a pivotal Hopf algebra.
\end{example}

\begin{example}
If $A=k$, $\varphi=id_C$, then the entwined pivotal morphism is a $k$-linear character $\mathfrak{g} \in C^\ast$, satisfies\\
$\left\{\begin{array}{l}
(1)~\mathfrak{g}(cd) = \mathfrak{g}(c)\mathfrak{g}(d);\\
(2)~\mathfrak{g}(1_C)=1_k;\\
(3)~\mathfrak{g}(c_1)S_C(c_2)=S^{-1}_C(c_1)\mathfrak{g}(c_2),
\end{array}\right.$
\\
for any $c,d \in C$, which implies $\mathfrak{g}$ is a \emph{copivot} (or a \emph{sovereign character}) on $C$ (see \cite{AAIT}, or \cite{jB}), and $C$ is a copivotal Hopf algebra.
\end{example}

\begin{example}
If $(C,A,\varphi,\mathfrak{g})$ is a pivotal entwined datum, and the following identity hold
$$
\sum a_\varphi \o (1_C)^\varphi = a \o 1_C, ~~~~\mbox{for~any~}a \in A,
$$
then $(A,\mathfrak{g}(1_C))$ is a pivotal Hopf algebra.
\end{example}

\begin{example}
If $(C,A,\varphi,\mathfrak{g})$ is a pivotal entwined datum, and the following identity hold
$$
\sum \varepsilon_A(a_\varphi) c^\varphi = \varepsilon_A(a)c,~~\mbox{for~any~}a \in A,~c\in C,
$$
then $(C, \varepsilon_A\ci \mathfrak{g})$ is a copivotal Hopf algebra.
\end{example}

\begin{example}\label{sw}
Let $k$ be a field and $H_4$ be the Sweedler's 4-dimensional Hopf algebra $H_4=k\{1_H,e,x,y| e^2=1_H,x^2=0, y=ex=-xe\}$ with the following structure
\begin{eqnarray*}
&\Delta(e)=e \o e,~\Delta(x)=x\o 1_H + e \o x,~\Delta(y)= y \o e + 1_H \o y,\\
&\varepsilon(e)=1,~\varepsilon(x)=\varepsilon(y)=0,~S(e)=e,~S(x)=-y,~S(y)=x.
\end{eqnarray*}
Recall from [\cite{jB}, Example 2.8] that $H_4$ is both a pivotal Hopf algebra and a copivotal Hopf algebra.
Moreover, the pivot in $H_4$ is $e$, and the copivot on $H_4$ is $I-E$, where $I,E$ are dual bases of $1_{H_4}$ and $e$.
\end{example}

\section{\textbf{The ribbon structure in the category of entwined modules}}
\def\theequation{5.\arabic{equation}}
\setcounter{equation} {0} \hskip\parindent

Now suppose that $(C,A,\vp,R)$ is a double quantum group over $k$, thus $\mathcal {C}^C_A(\varphi)$ is a braided category with the braiding which is defined by Eq.(2.1).
We also assume that $Nat(F \ci id, F \ci id)$ means the collection of natural transformations from the functor $F \ci id:{\mathcal {C}^C_A(\varphi)}\rightarrow Vec_k$ to itself. Then we have the following property.

\begin{proposition}\label{5.1}
There is a bijective map between the algebra $Nat(F \ci id, F \ci id)$ and $hom_k(C,A)$.
\end{proposition}

\begin{proof}
See [\cite{DHBP}, Theorem 2.1 and Proposition 2.4]

Actually, one can define a map $\Pi:Nat(F \ci id, F \ci id)\rightarrow hom_k(C,A)$ by
\begin{eqnarray*}
\Pi(\theta):C\rightarrow A, ~~~~c\mapsto\sum (\v_C \o A)\theta_{C \o A}(c \o 1_A),
\end{eqnarray*}
where $\theta \in Nat(F \ci id, F \ci id)$, $c \in C$.
Define $\Sigma:hom_k(C,A)\rightarrow Nat(F \ci id, F \ci id)$ by
$$
\Sigma(g)_M: M\rightarrow M,~~m\mapsto m_0\c g(m_1),
$$
where $g \in hom_k(C,A)$, $M \in \mathcal {C}^C_A(\varphi)$, $m \in M$.
It is a direct computation to check that $\Pi$ and $\Sigma$ are well-defined and inverse with each other.
\end{proof}

\begin{corollary}\label{5.2}
$Nat(F \ci id, F \ci id)$ and $Nat(F \ci id, F \ci F \ci {}^{\ast\ast}())$ are isomorphic with each other.
\end{corollary}

From now on, assume that $\theta \in Nat(F \ci id, F \ci id)$ and $g \in hom_k(C,A)$ are in correspondence with each other.

\begin{lemma}\label{5.3}
$\theta$ is a natural isomorphism if and only $g$ is invertible under the entwined convolution.
\end{lemma}

\begin{proof}
Straightforward from Proposition \ref{5.1}.
\end{proof}

\begin{lemma}\label{5.4}
For any $(M,\varrho_M, \rho^M) \in \mathcal {C}^C_A(\varphi)$, $\theta_M$ is $A$-linear if and only if $g$ satisfies
\begin{eqnarray}
g(c)a = \sum a_\vp g(c^\vp),~~~~~~~~~~\mbox{for~any~}~c\in C,~a \in A.
\end{eqnarray}
\end{lemma}

\begin{proof}
$\Leftarrow$: Since the following diagram
$$\aligned
\xymatrix{
m\otimes a \ar@{}"3,2"_{(E0)} \ar[r]^-{\rho^M \o id_A} \ar[dd]_-{\varrho_M} &  m_0\otimes m_1 \otimes a \ar[d]^{id_M \o \varphi}\ar[rr]^-{id_M \otimes g \o id_A} & & m_0\otimes g(m_1)\otimes a \ar@{}"2,4"_{Eq.(5.1)}
\ar[r]^-{\varrho_M \o id_A} &  m_0\c g(a_1) \otimes a \ar[dd]^-{\varrho_M} \\
& m_0 \otimes a_\vp \otimes {m_1}^\vp \ar[d]^-{\varrho_M \o id_C}\ar[rr]^-{id_M \o id_A \o g} & & m_0 \otimes a_\vp \otimes g({m_1}^\vp) \ar[d]^-{\varrho_M \o id_A}\ar"1,5"|{id_M \o m_A} & \\
m \c a \ar[r]^-{\rho^M } & m_0 \c a_\vp \otimes {m_1}^\vp \ar[rr]^-{id_M \otimes g} & & m_0 \c a_\vp \otimes g({m_1}^\vp) \ar@{}"1,5"_{Eq.(5.1)} \ar[r]^-{\varrho_M } & m_0 \c g(m_1)a
}
\endaligned$$
commutes for any $m \in M$, $a\in A$, $\theta_M$ is an $A$-module morphism.

$\Rightarrow$: Conversely, for the entwined module $A\o C$, since $\theta_{A\o C}$ is $A$-linear, then for any $c \in C$ and $a \in A$, we have the following commute diagram
$$\aligned
\xymatrix{
(1_A\otimes c) \o a \ar[rr]^-{\varrho_{A \o C}}  \ar[d]_-{\theta_{A \o C}\o id_A } & & \sum a_\vp \o c^\vp \ar[d]^{\theta_{A \o C}} \\
\sum ({g(c_2)}_\vp \o {c_1}^\vp) \o a \ar[rr]^-{\varrho_{A \o C}} & & \sum (a_\vp \o c^\vp)_0 \c g((a_\vp \o c^\vp)_1)
}
\endaligned$$
which implies
\begin{eqnarray*}
\sum a_\vp {g({c^\vp}_2)}_\psi \o {{c^\vp}_1}^\psi = \sum {g(c_2)}_\vp a_\psi \o {{c_1}^\vp}^\psi.
\end{eqnarray*}
Take $id_A \o \v_C$ to action at the both side of the above equation, we immediately get Eq.(5.1).
\end{proof}

\begin{lemma}\label{5.5}
For any $(M,\varrho_M, \rho^M) \in \mathcal {C}^C_A(\varphi)$, $\theta_M$ is $C$-colinear if and only if $g$ satisfies
\begin{eqnarray}
g(c_1) \o c_2 = \sum {g(c_2)}_\vp \o {c_1}^\vp ,~~~~~~~~~~\mbox{for~any~}~c\in C.
\end{eqnarray}
\end{lemma}

\begin{proof}
Be similar with Lemma \ref{5.4}.
\end{proof}

\begin{lemma}\label{5.6}
Suppose that $\theta$ is a natural transformation in $\mathcal {C}^C_A(\varphi)$, then $\theta$ is a twist if and only if for any $x,y \in C$, $g$ satisfies
\begin{eqnarray}
&&\Delta_A(g(xy)) = \nonumber\\
&&~~~~~~~~\sum g(x_1) {r^{(2)}(x_3 \o y_3)}_\psi R^{(1)}({y_2}^\vp \o {x_2}^\psi) \o g(y_1) {r^{(1)}(x_3 \o y_3)}_\vp R^{(2)}({y_2}^\vp \o {x_2}^\psi).
\end{eqnarray}
\end{lemma}

\begin{proof}
$\Leftarrow$: For any $M,N \in \mathcal {C}^C_A(\varphi)$ and $m \in M$, $n \in N$, we compute that
\begin{eqnarray*}
&&( \mathbf{C}_{N,M} \ci \mathbf{C}_{M,N} \ci (\theta_M \o \theta_N) ) (m \o n) \\
&=&( \mathbf{C}_{N,M} \ci \mathbf{C}_{M,N} ) ( m_0\c g(m_1) \o n_0\c g(n_1) ) \\
&\stackrel {(5.2)}{=}& \mathbf{C}_{N,M} (\sum  ( n_{00} \c {g(n_1)}_\vp \o m_{00} \c {g(m_1)}_\psi ) \c R({m_{01}}^\psi \o {n_{01}}^\vp ) )  \\
&\stackrel {(E1)}{=}&  \sum m_0 \c  g(m_2)_\vp R^{(2)}(m_{3} \o n_{3})_\phi r^{(1)} ( {n_1}^{\psi\chi} \o {m_1}^{\vp\phi} )  \o n_0 \c  g(n_2)_\psi R^{(1)}(m_{3} \o n_{3})_\chi r^{(2)} ( {n_1}^{\psi\chi} \o {m_1}^{\vp\phi})      \\
&\stackrel {(5.2)}{=}&    \sum m_0 \c  g(m_1) R^{(2)}(m_{3} \o n_{3})_\vp r^{(1)} ( {n_2}^{\psi} \o {m_2}^{\vp} )  \o n_0 \c  g(n_1) R^{(1)}(m_{3} \o n_{3})_\psi r^{(2)} ( {n_2}^{\psi} \o {m_2}^{\vp})      \\
&\stackrel {(5.3)}{=}&   (m_0 \o n_0) \c (g(m_1n_1)),
\end{eqnarray*}
which implies $\theta$ is a twist.

$\Rightarrow$: Conversely, for the entwined modules $C \o A$ and $A\o C$, since $\theta$ is a twist, then for any $x,y \in C$ we have
\begin{eqnarray*}
&&~~( \mathbf{C}_{A \o C,C \o A} \ci \mathbf{C}_{C \o A, A\o C} \ci (\theta_{C \o A} \o \theta_{A \o C}) ) ( (x \o 1_A) \o (1_A \o y) ) \\
&&= \theta_{C \o A, A \o C} ( (x \o 1_A) \o (1_A \o y) ),
\end{eqnarray*}
Since
\begin{eqnarray*}
&& ( \mathbf{C}_{A \o C,C \o A} \ci \mathbf{C}_{C \o A, A\o C} \ci (\theta_{C \o A} \o \theta_{A \o C}) ) ( (x \o 1_A) \o (1_A \o y) ) \\
&\stackrel {(5.2)}{=}&( \mathbf{C}_{A \o C,C \o A} \ci \mathbf{C}_{C \o A, A\o C} ) (  (x_1 \o g(x_2)) \o ( g(y_1) \o y_2 )  ) \\
&=& \mathbf{C}_{A \o C,C \o A} ( \sum ( g(y_1){R^{(1)}(x_3 \o y_3)}_\vp \o {y_2}^\vp ) \o ( x_1 \o g(x_2)R^{(2)}(x_3 \o y_3)  ) ) \\
&\stackrel {(E1)}{=}& \sum x_1 \o {g(x_3)}_\psi {R^{(2)}(x_4 \o y_3)}_\phi r^{(1)}( {{y_2}^\vp}_2 \o {x_2}^{\psi\phi} ) \o g(y_1){R^{(1)}(x_4 \o y_3)}_\vp {r^{(2)}( {{y_2}^\vp}_2 \o {x_2}^{\psi\phi} )}_\chi \o {{{y_2}^\vp}_1}^\chi,
\end{eqnarray*}
and
\begin{eqnarray*}
&&\theta_{C \o A, A \o C} ( (x \o 1_A) \o (1_A \o y) )\\
&=& x_1 \o {g(x_1y_1)}_1 \o {g(x_1y_1)}_2 \o y_1,
\end{eqnarray*}
we have
\begin{eqnarray*}
&&\sum x_1 \o {g(x_3)}_\psi {R^{(2)}(x_4 \o y_3)}_\phi r^{(1)}( {{y_2}^\vp}_2 \o {x_2}^{\psi\phi} ) \o g(y_1){R^{(1)}(x_4 \o y_3)}_\vp {r^{(2)}( {{y_2}^\vp}_2 \o {x_2}^{\psi\phi} )}_\chi \o {{{y_2}^\vp}_1}^\chi        \\
&&~~~~=x_1 \o {g(x_1y_1)}_1 \o {g(x_1y_1)}_2 \o y_1.
\end{eqnarray*}

Take $\v_C \o id_A id_A \o \v_C$ to action at the both side of the above equation, we immediately get Eq.(5.3).
\end{proof}

Recall from Theorem \ref{3.1}, we get that $\mathcal {C}^C_A(\varphi)$ is a rigid category. Then we get the following property.

\begin{lemma}\label{5.7}
$\theta$ is self-dual in $\mathcal {C}^C_A(\varphi)$ if and only if $g$ satisfies
\begin{eqnarray}
g(c) = \sum {(S^{-1}_A g S_C ( c^\vp))}_\vp,~~~~\mbox{for~any~}c\in C.
\end{eqnarray}
Or equivalently,
\begin{eqnarray}
g(c) = \sum  {a_i}_\vp  a^i ( S^{-1}_A g S_C ( {c}^\vp ) ),~~~~\mbox{for~any~}\g \in A^\ast,~c\in C,
\end{eqnarray}
where $a_i$ and $a^i$ are bases of $A$ and $A^\ast$ respectively, dual to each other.
\end{lemma}

\begin{proof}
$\Leftarrow$:
For any object $M \in \mathcal {C}^C_A(\varphi)$, suppose that $o_i$ and $o^i$ are dual bases of $M$ and $M^\ast$, $a_i$ and $a^i$ are dual bases of $A$ and $A^\ast$, $\m \in M^\ast$, $m \in M$, then we have
\begin{eqnarray*}
\theta_{M^\ast}(\mu)(m)&=& (\m_0 \c g(\m_1))(m)=\m_0( m \c S^{-1}_Ag(\m_1) ) \\
&\stackrel {(TR2)}{=}& \sum (\varrho_M)^\ast (\m_0) ( m \o S^{-1}_Ag(\m_1) )\\
&\stackrel {(TR1)}{=}& \sum \m_0( o_i \c a_i ) a^i(S^{-1}_Ag(\m_1)) o^i(m)\\
&=& \sum \m({o_i}_0 \c {a_i}_\vp) a^i( S^{-1}_AgS_C({{o_i}_1}^\vp) )o^i(m)\\
&=& \sum \m(m_0 \c {( S^{-1}_AgS_C({m_1}^\vp) )}_\vp)\\
&\stackrel {(5.4)}{=}& \m( m_0 \c g(m_1) ) = {(\theta_M)}^\ast(\m)(m).
\end{eqnarray*}
Thus $\theta$ is self-dual in $\mathcal {C}^C_A(\varphi)$.

$\Rightarrow$: Conversely, for $(C \o A)^\ast \in \mathcal {C}^C_A(\varphi)$, since $\theta$ is self-dual, we have
\begin{eqnarray*}
\theta_{{(C \o A)}^\ast}(\g \o \v_C)(c \o 1_A) = {(\theta_{C \o A})}^\ast(\g \o \v_C)(c \o 1_A), ~~\mbox{where~}\g\in A^\ast,~c \in C.
\end{eqnarray*}

For one thing, consider that
\begin{eqnarray*}
{(\theta_{C \o A})}^\ast(\g \o \v_C)(c \o 1_A)&=& (\g \o \v_C)( (c_1 \o 1_A) \c g(c_2) )\\
&=& \g(g(c)).
\end{eqnarray*}

For another, we compute
\begin{eqnarray*}
&&\theta_{{(C \o A)}^\ast}(\g \o \v_C)(c \o 1_A)\\
&\stackrel {(TR2)}{=}& \sum (\varrho_{C \o A})^\ast ({(\g \o \v_C)}_0) ( (c \o 1_A) \o S^{-1}_A g({(\g \o \v_C)}_1) )\\
&\stackrel {(TR1)}{=}& \sum {(\g \o \v_C)}_0( (c_i \o b_i) \c a_i ) a^i(S^{-1}_A g({(\g \o \v_C)}_1)) (b^i \o c^i) (c \o 1_A) \\
&=& \sum (\g \o \v_C) ( {c_i}_1 \o {(b_i a_i)}_\vp ) a^i ( S^{-1}_A g S_C ( {{c_i}_2}^\vp ) ) c^i(c) b^i(1_A) \\
&=& \sum \g( {a_i}_\vp ) a^i ( S^{-1}_A g S_C ( {c}^\vp ) ) = \g( (S^{-1}_A g S_C (c^\vp))_\vp ),
\end{eqnarray*}
where $c_i$ and $c^i$ are dual bases of $C$ and $C^\ast$, $a_i$ and $a^i$, $b_i$ and $b^i$ are two dual bases of $A$ and $A^\ast$.
Hence Eq.(5.4) holds.
\end{proof}

\begin{definition}\label{5.8}
Assume that $C, A$ are two Hopf algebras with bijective antipodes over a field $k$, and $(C, A, \varphi, R)$ is a double quantum group. If there exists a $k$-linear map $g:C\rightarrow A$, such that $g$ is invertible under the entwined comvolution, and Eqs.(5.1)-(5.4) are satisfied, then $g$ is called an \emph{entwined ribbon morphism} over $(C,A,\vp,R)$. Further, $(C,A,\varphi,R,g)$ is called a \emph{ribbon entwined datum}.
\end{definition}

Combining Proposition \ref{5.1} - Lemma \ref{5.7}, we can get our main theorem below.

\begin{theorem}\label{5.9}
Assume that $C,A$ are two Hopf algebras with bijective antipodes over $k$, $\varphi:C\o A \rightarrow A\o C$ and $R: C \o C\rightarrow A\o A $ are two $k$-linear maps such that $(C, A, \varphi, R)$ is a double quantum group. Then $\mathcal {C}^C_A(\varphi)$ is a ribbon category if and only if there is an entwined ribbon morphism $g\in Hom_k(C,A)$. Moreover, the ribbon structure $\theta$ in $\mathcal {C}^C_A(\varphi)$ is defined by
\begin{eqnarray*}
\theta_M:M\rightarrow M,~~\theta_M(m) = m_0 \c g(m_1)),~~\mbox{where~}~m \in M
\end{eqnarray*}
for any $(M,\theta_M,\rho^M) \in \mathcal {C}^C_A(\varphi)$.
\end{theorem}

\begin{example}
If $C=k$, $\varphi=id_A$, then the double quantum group $(C,A,\vp,R)$ becomes a quasitriangular Hopf algebra $(A,R)$, where $R$ means the $R$-matrix in $A$.
And the entwined ribbon morphism becomes an invertible element $g \in A$ satisfies\\
$\left\{\begin{array}{l}
(1)~g \mbox{~is~in~the~center~of~} A;\\
(2)~\Delta(g)= (g \o g)R_{21}R;\\
(3)~g=S(g),
\end{array}\right.$
\\
which implies $g$ is a usual ribbon element in $A$, thus $A$ is a ribbon Hopf algebra.
\end{example}

\begin{example}
If $A=k$, $\varphi=id_C$, then the double quantum group $(C,A,\vp,R)$ becomes a coquasitriangular Hopf algebra $(C,R)$.
And the entwined ribbon morphism is a convolution invertible $k$-linear character $g \in C^\ast$, satisfies\\
$\left\{\begin{array}{l}
(1)~g(c_1)c_2 = c_1g(c_2);\\
(2)~g(cd)= g(c_1)g(d_1)R(c_2 \o d_2)R(d_3 \o c_3) ;\\
(3)~g(c)=g(S(c)),
\end{array}\right.$
\\
for any $c,d \in C$, which implies $g$ is a coribbon form on $C$, thus $C$ is a coribbon Hopf algebra.
\end{example}

\begin{example}
Assume that $(C,A,\varphi,R,g)$ is a ribbon entwined datum. If the following identity hold
$$
\sum a_\varphi \o (1_C)^\varphi = a \o 1_C, ~~~~\mbox{for~any~}a \in A,
$$
then $(A,R(1_C \o 1_C))$ is a quasitriangular Hopf algebra. Further,
$(A,g(1_C))$ is a ribbon Hopf algebra.

If the following identity hold
$$
\sum \varepsilon_A(a_\varphi) c^\varphi = \varepsilon_A(a)c, ~~~\mbox{for~any~}c \in C,~a\in A,
$$
then $(C, (\v_A\o \v_A) \ci R)$ is a coquasitriangular Hopf algebra. Further,
 $(C, \varepsilon_A\ci g)$ is a coribbon Hopf algebra.
\end{example}

\section{\textbf{Entwined smash product}}

\vskip 0.5cm
\subsection{\textbf{The entwined smash product}}
\def\theequation{6.1.\arabic{equation}}
\setcounter{equation} {0} \hskip\parindent
\vskip 0.5cm

\begin{definition}
Let $A$, $B$ be algebras in a monoidal category $\mathcal{C}$. A morphism $\Phi:B \o A\rightarrow A \o B$ in $\mathcal{C}$ is called an \emph{algebra distributive law} if $\Phi$ satisfying
$$\aligned
\xymatrix
{
B \o B \o A \ar[d]_-{id_B \o \Phi}\ar"1,3"^-{m_B \o id_A} & &B \o A \ar[d]^-{\Phi}\\
B \o A \o B \ar[r]_-{\Phi \o id_B} & A \o B \o B \ar[r]_-{id_A \o m_B} & A \o B ,
}~~~
\xymatrix{
A \ar[r]^-{\eta_B \o id_A} \ar[dr]_-{id_A \o \eta_B} & B \o A \ar[d]^{\Phi}  \\
 & A \o B , }
\endaligned$$
$$\aligned
\xymatrix
{
B \o A \o A \ar[d]_-{\Phi \o id_A}\ar"1,3"^-{id_B \o m_A} & & B \o A \ar[d]^-{\Phi}\\
A \o B \o A \ar[r]_-{ id_A \o\Phi} & A \o A \o B \ar[r]_-{m_A \o id_B} & A \o B ,
}~~~
\xymatrix{
B \ar[dr]_-{\eta_A \o id_B} \ar[r]^-{id_B \o \eta_A} & B \o A \ar[d]^{\Phi}  \\
 & A \o B . }
\endaligned$$
\end{definition}

Be similar with [\cite{sgz1}, Theorem 8], we have the following property.

\begin{lemma} \label{6.0}
Let $C$ be a finite dimensional coalgebra and $A$ a finite dimensional algebra over $k$. Then give an entwining map $\varphi: C \o A\rightarrow A \o C$ is identified to give an algebra distributive law $\Phi : A \o C^{\ast op}\rightarrow C^{\ast op} \o A$.
\end{lemma}

\begin{proof}
If there is an entwining map $\varphi: c\o a \mapsto \sum a_\varphi \o c^\varphi$, one can define a linear map $\Phi : A \o C^{\ast op}\rightarrow C^{\ast op} \o A$ by
$$
\Phi(a \o p)=\sum p^\Phi \o a_\Phi := \sum p({e_i}^\varphi) e^i \o a_\varphi,
$$
where $c \in C$, $a \in A$, $p \in C^{\ast}$, $e_i$ and $e^i$  are dual bases of $C$ and $C^\ast$. It is straightforward to see that $\Phi$ is an algebra distributive law.

Conversely, if is an algebra distributive law $\Phi: a\o p \mapsto \sum p^\Phi \o a_\Phi$, one can define a linear map $\varphi : C \o A\rightarrow A \o C$ by
$$
\varphi(c \o a)=\sum a_\varphi \o c^\varphi := \sum {e^i}^\Phi(c) a_\Phi \o e_i.
$$
Also it can be easily checked that $\varphi$ is an entwining map.
\end{proof}

Recall from [\cite{sgz}, Theorem 2.5], if there is an algebra distributive law $\Phi:B \o A\rightarrow A \o B$, then $(A \o B, (m_A \o m_B)\ci(id_A \o \Phi \o id_B), \eta_A \o \eta_B)$ is also an algebra.

Now we supposed that $(C,A, \varphi)$ is a monoidal entwining datum over $k$ where $C, A$ are two Hopf algebras with bijective antipodes.

\begin{definition} \label{6.1}
The \emph{entwined smash product} ${C^{\ast}}^{op} \o A$ of the entwining structure $(C,A,\varphi)$, in a form containing ${C^{\ast}}^{op}$ and $A$, is a Hopf algebra
with the following structures:

$\bullet$ the multiplication $\widehat{m}$ is given by
\begin{eqnarray*}
(p \o a)(q \o b) := \sum p \ast^{op} e^i \o a_\varphi b q({e_i}^\varphi)= \sum e^i \ast p \o a_\varphi b q({e_i}^\varphi),
\end{eqnarray*}
where $a,b \in A$, $p,q \in {C^{\ast}}^{op}$, $e_i$ and $e^i$  are dual bases of $C$ and $C^\ast$;

$\bullet$ the unit is $\widehat{\eta}(1_k)=\varepsilon_C \o 1_A$;

$\bullet$ the comultiplication is given by
$$\widehat{\Delta}(p\o a) := (p_1 \o a_1) \o (p_2 \o a_2);$$

$\bullet$ the counit is given by
$$\widehat{\varepsilon}(p \o a) := p(1_C)\varepsilon_A(a);$$

$\bullet$ the antipode is given by
$$\widehat{S}(p \o a):= \sum p(S^{-1}_C({e_i}^\varphi)) e^i \o {S_A(a)}_\varphi.$$
\end{definition}

\begin{proof}
Firstly, since Lemma 6.2, ${C^{\ast}}^{op} \o A$ is an algebra.

Next we will show ${C^{\ast}}^{op} \o A$ is a bialgebra.
Obviously, ${C^{\ast}}^{op} \o A$ is a coalgebra under the given comultiplication. We only need check that $\widehat{\D}$ and $\widehat{\v}$ are algebra maps.

For $p,q \in {C^{\ast}}^{op}$, $a,b\in A$, we compute
\begin{eqnarray*}
&&~~~\widehat{\D}(p \o a)\widehat{\D}(q \o b)\\
&&=((p_1 \o a_1) \o (p_2 \o a_2))((q_1 \o b_2) \o (q_1 \o b_2))\\
&&=\sum p_1 \ast^{op} e^i \o a_{1\varphi}b_1 q_1({e_i}^\varphi) \o p_2 \ast^{op} o^i \o a_{2\psi}b_2 q_2({o_i}^\psi).
\end{eqnarray*}
Thus for any $c,d \in C$, we have
\begin{eqnarray*}
&&~~~\sum (p_1 \ast^{op} e^i) (c) \o a_{1\varphi}b_1 q_1({e_i}^\varphi) \o (p_2 \ast^{op} o^i)(d) \o a_{2\psi}b_2 q_2({o_i}^\psi)\\
&&=\sum p(c_2d_2)\o a_{1\varphi}b_1 \o q({c_1}^\varphi {d_1}^\psi) \o a_{2\psi}b_2 .
\end{eqnarray*}
Also we have
\begin{eqnarray*}
\widehat{\D}( (p \o a) (q \o b))&=& \widehat{\D} ( \sum p \ast^{op} e^i \o a_\varphi b q({e_i}^\varphi))\\
&=&\sum p_1 \ast^{op} {e^i}_1 \o a_{\varphi1}b_1 \o p_2 \ast^{op} {e^i}_2 \o a_{\varphi2}b_2 q({e_i}^\varphi),
\end{eqnarray*}
Then for $c,d \in C$, we obtain
\begin{eqnarray*}
&&\sum (p_1 \ast^{op} {e^i}_1)(c) \o a_{\varphi1}b_1 \o (p_2 \ast^{op} {e^i}_2)(d) \o a_{\varphi2}b_2 q({e_i}^\varphi)\\
&=& \sum p(c_2 d_2) \o a_{\varphi1}b_1 \o  q({(c_1 d_1)}^\varphi) \o a_{\varphi2}b_2 \\
&\stackrel {(E5)}{=}& \sum p(c_2d_2)\o a_{1\varphi}b_1 \o q({c_1}^\varphi {d_1}^\psi) \o a_{2\psi}b_2,
\end{eqnarray*}
which implies $\widehat{\D}((p \o a)(q \o b)) = \widehat{\D}(p \o a)\widehat{\D}(q \o b)$ .

Since  $\widehat{\v}$  preserves multiplication, ${C^{\ast}}^{op} \o A=({C^{\ast}}^{op} \o A, \widehat{m},\v_C \o 1_A,\widehat{\D},\widehat{\v})$ is a bialgebra.

In order to prove $\widehat{S}$ is the antipode of ${C^{\ast}}^{op} \o A$, for one thing, we compute
\begin{eqnarray*}
&&~~~\widehat{S}((p \o a)_1)(p \o a)_2\\
&&=\sum p_1(S_C^{-1}({e_i}^\varphi)) (e^i \o {S_A(a_1)}_\varphi) (p_2 \o a_2)\\
&&=\sum p(S_C^{-1}({e_i}^\varphi) {o_i}^\psi) e^i \ast^{op} o^i \o {S_A(a_1)}_{\varphi\psi} a_2.
\end{eqnarray*}
For any $c \in C$, we have
\begin{eqnarray*}
&&\sum p(S_C^{-1}({e_i}^\varphi) {o_i}^\psi) (e^i \ast^{op} o^i)(c) \o {S_A(a_1)}_{\varphi\psi} a_2\\
&=&\sum p(S_C^{-1}({c_2}^\varphi) {c_1}^\psi) \o {S_A(a_1)}_{\varphi\psi} a_2\\
&\stackrel {(E2)}{=}&\sum p(S_C^{-1}({c^\varphi}_2) {c^\varphi}_1) \o {S_A(a_1)}_{\varphi} a_2\\
&=& p(1_C)\v_C(c)\o S_A(a_1) a_2 = p(1_C)\v_C(c)\o \v_A(a)1_A.
\end{eqnarray*}
Thus $\widehat{S} \ast id = \widehat{\eta} \widehat{\v}$. Similarly, one can show that $id \ast \widehat{S}= \widehat{\eta} \widehat{\v}$. Hence $({C^{\ast}}^{op} \o A,\widehat{S})$ is a Hopf algebra.
\end{proof}

Next we will present the proof of Eq.(3.2).

\begin{proof}
For any $a \in A$, $c \in C$, $p \in C^{\ast}$, $\g \in A^\ast$, since $\widehat{S}$ is the antipode of ${C^{\ast}}^{op} \o A$, we have
\begin{eqnarray*}
\widehat{S}((\v_c\o a) (p \o 1_A)) = \widehat{S}(p \o 1_A)\widehat{S}(\v_c\o a).
\end{eqnarray*}

For one thing, we compute
\begin{eqnarray*}
\widehat{S}((\v_c\o a) (p \o 1_A)) &=& \widehat{S} (e^i \o a_\varphi) p({e_i}^\varphi) \\
&=& \sum p({S^{-1}_C({o_i}^\psi)}^\varphi)  o_i \o {S_A(a_\varphi)}_\psi,
\end{eqnarray*}
where $e_i$ ($o_i$) and $e^i$ ($o^i$) are dual bases of $C$ and $C^\ast$ respectively.

For another, we have
\begin{eqnarray*}
\widehat{S}(p \o 1_A)\widehat{S}(\v_c\o a) &=&(p(S^{-1}_C(e_i))e^i \o 1_A ) ( \v_C \o S_A(a) ) \\
&=& \sum p(S^{-1}_C(e_i)) e^i \o S_A(a).
\end{eqnarray*}
Thus $\sum p({S^{-1}_C({o_i}^\psi)}^\varphi)  o_i \o {S_A(a_\varphi)}_\psi = \sum p(S^{-1}_C(e_i)) e^i \o S_A(a)$.
Indeed, we can easily get
\begin{eqnarray*}
\sum p({S^{-1}_C({o_i}^\psi)}^\varphi)  o_i(c) \o \g({S_A(a_\varphi)}_\psi) = \sum p(S^{-1}_C(e_i)) e^i(c) \o \g(S_A(a)),
\end{eqnarray*}
i.e.
\begin{eqnarray*}
\sum p({S^{-1}_C({c}^\psi)}^\varphi) \o \g({S_A(a_\varphi)}_\psi) = \sum p(S^{-1}_C(c)) \o \g(S_A(a)),
\end{eqnarray*}
which implies Eq.(3.2).
\end{proof}

Be similar with [\cite{sgz1}, Theorem 9], we have the following property.

\begin{proposition} \label{etD}
The category of entwined modules $\mathcal {C}^C_A(\varphi)$ is monoidal isomorphic to the representation category of ${C^{\ast}}^{op} \o A$.
\end{proposition}

\begin{proof}
For any object $(M,\theta_M,\rho^M)$, and morphism $\lambda:M\rightarrow N$ in $\mathcal {C}^C_A(\varphi)$, one can define a functor $\Gamma$ from $\mathcal {C}^C_A(\varphi)$ to the category of right ${C^{\ast}}^{op} \o A$-modules via
\begin{eqnarray*}
\Gamma(M):=M\mbox{~as~}k\mbox{-modules~},~~~~\Gamma(f):=f,
\end{eqnarray*}
where the ${C^{\ast}}^{op} \o A$-module structure on $M$ is given by
\begin{eqnarray*}
m \leftharpoonup (p\o a):= p(m_1)m_0 \c a, ~~~~\mbox{for~all~}m \in M,~p\in C^{\ast},~a \in A.
\end{eqnarray*}

First of all, we claim that $\Gamma$ is well-defined.
In fact,for any $m \in M$, $p,q\in C^{\ast}$, $a,b \in A$, we have
\begin{eqnarray*}
m\leftharpoonup (\v_C \o 1_A)=\v_C(m_1)m_0 \c 1_A = m.
\end{eqnarray*}

Also, we can get
\begin{eqnarray*}
(m\leftharpoonup (p \o a))\leftharpoonup (q\o b)&=&p(m_1)(m_0 \c a)\leftharpoonup (q\o b) \\
&=& \sum p(m_2)e^i(m_1) m_0 \c a_\varphi b q({e_i}^\varphi) \\
&=& m\leftharpoonup (p \o a)(q \o b),
\end{eqnarray*}
where $e_i$ and $e^i$ are dual bases of $C$ and $C^\ast$ respectively.
Hence $(M,\leftharpoonup)$ is a right ${C^{\ast}}^{op} \o A$-module.

For the morphism $\lambda:M\rightarrow N$, it is a direct computation to check $\Gamma(\lambda)$ is ${C^{\ast}}^{op} \o A$-linear. Thus $\Gamma$ is well-defined.

Conversely, we define the functor $\Lambda$ from the representation category of ${C^{\ast}}^{op} \o A$ to $\mathcal {C}^C_A(\varphi)$ by
\begin{eqnarray*}
\Lambda(U):=U\mbox{~as~}k\mbox{-modules~},~~~~\Lambda(\lambda):=\lambda,
\end{eqnarray*}
where $(U, \leftharpoonup)$ is a right ${C^{\ast}}^{op} \o A$-module, $\lambda:U\rightarrow V$ is a morphism of ${C^{\ast}}^{op} \o A$-modules.
Further, the $A$-action on $U$ is defined by
\begin{eqnarray*}
u \c a:= u\leftharpoonup (\v_C \o a), ~~~~\mbox{for~any~}u\in U,~~a\in A,
\end{eqnarray*}
and the $C$-coaction on $U$ is given by
\begin{eqnarray*}
\rho^U(u)=u_{0} \o u_{1}:= \sum (u\leftharpoonup (e^i \o 1_A)) \o e_i.
\end{eqnarray*}

Next we will show that $\Lambda$ is well defined. It is straightforward to show $(U, \c)$ is an $A$-module and $(U,\rho^U)$ is a $C$-comodule. We only check $U$ satisfies Diagram (E0).

Since for any $a \in A$, we have
\begin{eqnarray*}
\rho^U(u \c a) &=& \sum ( u\c a\leftharpoonup (e^i \o 1_A)) \o e_i \\
&=& \sum ( u\leftharpoonup (e^i({o_i}^\varphi)o^i \o a_\varphi)) \o e_i \\
&=& \sum ( u\leftharpoonup (e^i \o 1_A)(\v_C \o a_\varphi)) \o {e_i}^\varphi \\
&=& \sum u_0 \c a_\varphi \o {u_1}^\varphi,
\end{eqnarray*}
hence $U \in \mathcal {C}^C_A(\varphi)$.

Since $\Lambda(\lambda):U\rightarrow V$  are both $A$-linear and $C$-colinear, $\Lambda$ is well-defined, as desired.

Obviously $\Gamma$ is a strict monoidal functor, and $\Lambda$ is the inverse of $\Gamma$. This completes the proof.
\end{proof}

\begin{remark}
Be similar with Lemma 6.2 and Proposition \ref{etD}, for any finite dimensional $k$-algebras $A$ and $B$, if $\Phi:B \o A\rightarrow A \o B$, $b \o a\mapsto \sum a^\Phi \o b_\Phi$, is an algebra distributive law, then there is an entwining map $\varphi: A^{\ast cop} \o B \rightarrow B \o A^{\ast cop}$, defined by
$$
\varphi(\gamma \o b):= \sum \gamma({e_i}^\Phi) b_\Phi \o e^i, ~~\mbox{where~}b \in B,~~\gamma \in A^\ast,~~e_i\mbox{~and~}e^i \mbox{~are~dual~bases~of~}A\mbox{~and~}A^\ast.
$$

Conversely, if there an entwining map $\varphi: A^{\ast cop} \o B \rightarrow B \o A^{\ast cop}$, $\gamma \o b \mapsto \sum b_\varphi \o \gamma^\varphi$, then one can define an algebra distributive law $\Phi:B \o A\rightarrow A \o B$ via
$$
\Phi(b \o a):= \sum {e^i}^\varphi(a) e_i \o b_\varphi.
$$

Moreover, the category of $A \o B$-modules is identified to $\mathcal{C}^{A^{\ast cop}}_B(\varphi)$, the category of entwined modules.
\end{remark}

\begin{example}
(1). If we define $\varphi:H \o H\rightarrow H \o H$ by $\varphi(x \o y) = y_1 \o xy_2$, where $x, y \in H$ and $H$ is a finite dimensional Hopf algebra over $k$, then the category $\mathcal {C}^H_H(\varphi)$ is the category of Hopf modules. Recall from Lemma 4.2 and Theorem \ref{etD}, if we define the following multiplication on $H^{\ast op} \o H$
$$
(\delta \o a)(\gamma \o b):= \gamma(?a_2)\delta \o a_1b,
$$
then $H^{\ast op} \o H$ is an associative algebra, and the category of Hopf modules of $H$ is identified to the category of $H^{\ast op} \o H$-modules.

(2). Let $H$ be a finite dimensional Hopf algebra and $A$ a finite dimensional left $H$-module algebra. Recall that the multiplication on $A\sharp H$, the usual smash product of $A$ and $H$ is
$$
(a\sharp x)(b \sharp y)= a(x_1 \c b) \sharp x_2 y, ~~\mbox{where~}a,b \in A,~~x,y \in H,
$$
which implies there is an algebra distributive law $\Phi:H \o A\rightarrow A \o H$, $\Phi(h \o a)= h_1 \c a \o h_2$.
Thus from Remark 4.6, there exists an entwining map $\varphi:A^{\ast cop} \o H \rightarrow H \o A^{\ast cop}$, defined by
$$
\varphi(\gamma \o h):= \sum \gamma(h_1 \c e_i)h_2 \o e^i, ~~\mbox{where~}h \in H,~~\gamma \in A^\ast,~~e_i\mbox{~and~}e^i \mbox{~are~dual~bases~of~}A\mbox{~and~}A^\ast.
$$
Furthermore, the representations of $A\sharp H$ is isomorphic to the category $\mathcal{C}^{A^{\ast cop}}_H(\varphi)$.

\end{example}

\begin{theorem}
There exists a $k$-linear map $\mathfrak{g}:C\rightarrow A$ such that
 $(C,A,\varphi,\mathfrak{g})$ is a pivotal entwined datum if and only if ${C^{\ast}}^{op} \o A$ is a pivotal Hopf algebra.
\end{theorem}

\begin{proof}
$\Rightarrow$:
If $(C,A,\varphi,\mathfrak{g})$ is a pivotal entwined datum, then the pivot element in ${C^{\ast}}^{op} \o A$ is $\sum e^i \o \mathfrak{g}(e_i)$, where $e_i$ and $e^i$ are dual bases of $C$ and $C^\ast$ respectively.

$\Leftarrow$:
Conversely, if ${C^{\ast}}^{op} \o A$ is a pivotal Hopf algebra with the pivot $T=\sum T^{(1)} \o T^{(2)} \in {C^{\ast}}^{op} \o A$, then the entwined pivotal morphism of $\mathcal {C}^C_A(\varphi)$ is $c\mapsto \sum T^{(1)}(c) T^{(2)}$.
\end{proof}

\begin{theorem}
Suppose that $(C,A,\vp,R)$ is a double quantum group where $R$ is a map from $C \o C$ to $A \o A$. Then there exists a $k$-linear map $g:C\rightarrow A$ such that
 $(C,A,\varphi,R,g)$ is a ribbon entwined datum if and only if ${C^{\ast}}^{op} \o A$ is a ribbon Hopf algebra.
\end{theorem}

\begin{proof}
$\Rightarrow$:
If $(C,A,\varphi,R,g)$ is a ribbon entwined datum, then the $R$-matrix of ${C^{\ast}}^{op} \o A$ is $\sum c^i \o R^{(2)}(c_i \o e_i) \o e^i \o R^{(1)}(c_i \o e_i)$, where $e_i$ and $e^i$, $c_i$ and $c^i$ are all dual bases of $C$ and $C^\ast$ respectively. Further, the ribbon element in ${C^{\ast}}^{op} \o A$ is $\sum e^i \o g(e_i)$, where $e_i$ and $e^i$ are dual bases of $C$ and $C^\ast$ respectively.

$\Leftarrow$:
Conversely, if ${C^{\ast}}^{op} \o A$ is a ribbon Hopf algebra with the ribbon element $L=\sum L^{(1)} \o L^{(2)} \in {C^{\ast}}^{op} \o A$, then the entwined ribbon morphism of $\mathcal {C}^C_A(\varphi)$ is $c\mapsto \sum L^{(1)}(c) L^{(2)}$.
\end{proof}

\vskip 0.5cm
\subsection{\textbf{The dual case}}
\def\theequation{4.2.\arabic{equation}}
\setcounter{equation} {0} \hskip\parindent
\vskip 0.5cm

The definition and results in this section are dual to the corresponding results in Section 4.1, so we will not give the complete proof.

\begin{definition}
Let $C$, $D$ be coalgebras in a monoidal category $\mathcal{C}$. A morphism $\Psi:C\o D\rightarrow D \o C$ in $\mathcal{C}$ is called a \emph{coalgebra distributive law} if $\Psi$ satisfying
$$\aligned
\xymatrix
{
C \o D \ar[d]_-{id_C \o \Delta_D}\ar"1,3"^-{\Psi} & & D \o C \ar[d]^-{\Delta_D \o id_C}\\
C \o D \o D \ar[r]_-{\Psi \o id_D} & D\o C \o D \ar[r]_-{id_D \o \Psi} & D \o D \o C ,
}~~~
\xymatrix{
C \o D \ar[r]^-{\Psi} \ar[dr]_-{id_C \o \varepsilon_D} & D \o C \ar[d]^{\varepsilon_D \o id_C}  \\
 & C , }
\endaligned$$
$$\aligned
\xymatrix
{
C \o D \ar[d]_-{\Delta_C \o id_D}\ar"1,3"^-{\Psi} & & D \o C \ar[d]^-{id_D \o \Delta_C}\\
C \o C \o D \ar[r]_-{id_C \o \Psi } & C\o D \o C \ar[r]_-{\Psi \o id_C } & D \o C \o C ,
}~~~
\xymatrix{
C \o D \ar[r]^-{\Psi} \ar[dr]_-{\varepsilon_C \o id_D} & D \o C \ar[d]^{id_D \o\varepsilon_C}  \\
 & D . }
\endaligned$$
\end{definition}

Be similar with [\cite{sgz1}, Theorem 12], we have the following property.

\begin{lemma}
Let $C$ be a finite dimensional coalgebra and $A$ a finite dimensional algebra over $k$. Then give an entwining map $\varphi: C \o A\rightarrow A \o C$ is identified to give a coalgebra distributive law $\Psi : A^{\ast cop} \o C\rightarrow C \o A^{\ast cop}$.
\end{lemma}

Now supposed that $(C,A, \varphi)$ is a monoidal entwining datum over $k$ where $C, A$ are two Hopf algebras with bijective antipodes.

\begin{definition} \label{5.01}
The \emph{entwined smash coproduct} ${A^{\ast}}^{cop}\o C $ of $(C,A,\varphi)$, in a form containing ${A^{\ast}}^{cop}$ and $C$, is a Hopf algebra
with the following structures:

$\bullet$ the multiplication $\overline{m}$ is given by
\begin{eqnarray*}
(\g \o c)(\delta \o d) := \g \ast \delta \o ab,
\end{eqnarray*}
where $c,d \in A$, $\g,\delta \in {A^{\ast}}^{cop}$;

$\bullet$ the unit is $\overline{\eta}(1_k)=\varepsilon_A \o 1_C$;

$\bullet$ the comultiplication is given by
$$\overline{\Delta}(\g\o c) := \sum (\g_1({e_i}_\varphi) \g_2 \o {c_1}^\varphi) \o (e^i \o c_2),$$
where $e_i$ and $e^i$  are dual bases of $A$ and $A^\ast$

$\bullet$ the counit is given by
$$\overline{\varepsilon}(\g \o c) := \g(1_A)\varepsilon_C(c);$$

$\bullet$ the antipode is given by
$$\overline{S}(\g \o c):= \sum \g({e_i}_\varphi) S_{A^\ast}^{-1}(e^i) \o S_C(c^\varphi).$$
\end{definition}

The following is the proof of Eq.(3.1).

\begin{proof}
For any $a \in A$, $c \in C$, $p \in C^{\ast}$, $\g \in A^\ast$, since $\overline{S}$ is the antipode of ${A^{\ast}}^{cop} \o C$, we have
\begin{eqnarray*}
(\overline{S} \o\overline{S}) (\overline{\Delta}^{cop}(\g \o c)) = \overline{\Delta}(\overline{S}(\g\o c)).
\end{eqnarray*}

For one thing, we compute
\begin{eqnarray*}
&&(\overline{S} \o\overline{S}) (\overline{\Delta}^{cop}(\g \o c)) \\
&=& \sum (\overline{S} \o\overline{S}) ((e^i(o_{i\phi}) S^{-1}_{A^\ast}(o^i) \o S_C({c_2}^{\phi})) \o (\g_1(e_{i\varphi}) \g_2(x_{i\psi}) S^{-1}_{A^\ast}(x^i) \o S_C({c_1}^{\varphi\psi})) )\\
&=& \sum \g(o_{i\phi\varphi} x_{i\psi}) S^{-1}_{A^\ast}(o^i) \o S_C({c_2}^{\phi})\o S^{-1}_{A^\ast}(x^i) \o S_C({c_1}^{\varphi\psi}),
\end{eqnarray*}
where $e_i$,$o_i$,$x_i$ and $e^i$,$o^i$,$x^i$ are all dual bases of $A$ and $A^\ast$ respectively.

For another, we have
\begin{eqnarray*}
\overline{\Delta}(\overline{S}(\g\o c))= \sum \g(e_{i\varphi}) S^{-1}_{A^\ast}({e^i}_2)(o_{i\psi}) S^{-1}_{A^\ast}({e^i}_1) \o {S_C({c^\varphi}_2)}^{\psi} \o o^i \o S_C({c^\varphi}_1).
\end{eqnarray*}

Indeed, since
\begin{eqnarray*}
&& (p \o \v_C) ( \sum \g(o_{i\phi\varphi} x_{i\psi}) S^{-1}_{A^\ast}(o^i)(1_A) \o S_C({c_2}^{\phi})\o S^{-1}_{A^\ast}(x^i)(a) \o S_C({c_1}^{\varphi\psi}) )\\
&=& \sum \g( {S^{-1}_{A}(a)}_{\psi}) p(S_C(c_2)) \v_C( S_C({c_1}^{\psi})) \\
&=& \g( S^{-1}_{A}(a)) p(S_C(c)),
\end{eqnarray*}
and
\begin{eqnarray*}
&& (p \o \v_C) ( \sum \g(e_{i\varphi}) S^{-1}_{A^\ast}({e^i}_2)(o_{i\psi}) S^{-1}_{A^\ast}({e^i}_1)(1_A) \o {S_C({c^\varphi}_2)}^{\psi} \o o^i(a) \o S_C({c^\varphi}_1) )\\
&=& (p \o \v_C) ( \sum \g(e_{i\varphi}) S^{-1}_{A^\ast}(e^i)(a_{\psi}) {S_C({c^\varphi}_2)}^{\psi} \o S_C({c^\varphi}_1) ) \\
&=& \sum \g({S^{-1}_A(a_{\psi} )}_{\varphi}) p({S_C({c^\varphi})}^{\psi}),
\end{eqnarray*}
Eq.(3.1) holds.
\end{proof}

Be similar with [\cite{sgz1}, Theorem 13], we have the following property.

\begin{proposition} \label{5.02}
$\mathcal {C}^C_A(\varphi)$ is monoidal isomorphic to the corepresentation category of ${A^{\ast}}^{cop}\o C$.
\end{proposition}

\begin{theorem}
There exists a $k$-linear map $\mathfrak{g}:C\rightarrow A$ such that
 $(C,A,\varphi,\mathfrak{g})$ is a pivotal entwined datum if and only if ${A^{\ast}}^{cop}\o C$ is a copivotal Hopf algebra.
\end{theorem}

\begin{proof}
If $(C,A,\varphi,\mathfrak{g})$ is a pivotal entwined datum, then the copovit on ${A^{\ast}}^{cop}\o C$ is $\g \o c\mapsto \g(\mathfrak{g}(c))$.

Conversely, if ${A^{\ast}}^{cop}\o C$ is a copivotal Hopf algebra with the pivotal character $\Gamma \in {({A^{\ast}}^{cop}\o C)}^\ast$, then the entwined pivotal morphism of $\mathcal {C}^C_A(\varphi)$ is $c\mapsto \sum \Gamma(e^i \o c)e_i$, where $e_i$ and $e^i$ are dual bases of $A$ and $A^\ast$ respectively.
\end{proof}

\begin{theorem}
Suppose that $(C,A,\vp,R)$ is a double quantum group where $R$ is a map from $C \o C$ to $A \o A$.
Then there exists a $k$-linear map $g:C\rightarrow A$ such that
 $(C,A,\varphi,R,g)$ is a ribbon entwined datum if and only if ${A^{\ast}}^{cop}\o C$ is a coribbon Hopf algebra.
\end{theorem}

\begin{proof}
If $(C,A,\varphi,R,g)$ is a ribbon entwined datum, then the
coquasitriangular structure on ${A^{\ast}}^{cop}\o C$ is
\begin{eqnarray*}
\zeta:({A^{\ast}}^{cop}\o C) \o ({A^{\ast}}^{cop}\o C) \rightarrow k,~~\zeta((\g' \o c) \o (\g \o d))\mapsto (\g' \o \g) R(c \o d).
\end{eqnarray*}
And the ribbon character on ${A^{\ast}}^{cop}\o C$ is $\g \o c\mapsto \g(g(c))$.

Conversely, if ${A^{\ast}}^{cop}\o C$ is a coribbon Hopf algebra with the coribbon form $\Theta \in {({A^{\ast}}^{cop}\o C)}^\ast$, then the entwined ribbon morphism of $\mathcal {C}^C_A(\varphi)$ is $c\mapsto \sum \Theta(e^i \o c)e_i$, where $e_i$ and $e^i$ are dual bases of $A$ and $A^\ast$ respectively.
\end{proof}

\section{\textbf{Applications}}

\vskip 0.5cm
\subsection{\textbf{Generalized Long dimodules}}
\def\theequation{7.1.\arabic{equation}}
\setcounter{equation} {0} \hskip\parindent
\vskip 0.5cm

Suppose that $H$ and $B$ are both finite dimensional Hopf algebras over $k$, $M$ is at the same time a right $H$-module and a right $B$-comodule. Recall that $M$ is called a \emph{generalized right-right Long dimodule} (see \cite{fl}) if
$$\rho( m\cdot h) = \sum m_{(0)}\cdot h\otimes m_{(1)}$$
for all $m \in M$ and $h \in H$. The category of generalized right-right Long dimodules and $H$-linear $B$-colinear homomorphisms is denoted by ${\mathcal{L}}_H^B$.
If we define $\dot{\varphi} : B\o H \rightarrow H\o B$ as the flip map in $Vec_k$, then obviously $(B,H,\dot{\varphi})$ is a monoidal entwining datum, and ${\mathcal{L}}^B_H = \mathcal{C}^B_H(\dot{\varphi})$.

Then from Theorem \ref{etD}, we immediately get that the category of generalized Long dimodules is identified to the representations of the Hopf algebra $B^{\ast op} \o H$. Here the bialgebra structure of $B^{\ast op} \o H$ is the ordinary bialgebra structure which is induced by the tensor product of $B^{\ast op}$ and $H$, and the antipode is defined by
$$
\overline{S}(p \o a) = S^{-1}_{B^\ast}(p) \o S_H(a),~~\mbox{where~}p \in B^\ast,~~a \in H.
$$

We can get the following results from Theorem \ref{4.6} and Theorem \ref{5.9}.

\begin{theorem} \label{7.1}
${\mathcal{L}}^B_H$ is a pivotal category if and only if there is a
$k$-linear map $\mathfrak{g}:B\rightarrow H$, such that for any $a,b \in B$ and $h \in H$, the following conditions hold:
\\
$\left\{\begin{array}{l}
(LP1)~~\Delta(\mathfrak{g}(ab))=\mathfrak{g}(a) \o \mathfrak{g}(b);\\
(LP2)~~\varepsilon_H(\mathfrak{g}(1_B))=1_k;\\
(LP3)~~\mathfrak{g}(a)S_H(h) = S^{-1}_H(h) \mathfrak{g}(a);\\
(LP4)~~\mathfrak{g}(a_1) \o S_B(a_2) =  \mathfrak{g}(a_2) \o S^{-1}_B(a_1).
\end{array}\right.$
\\
\end{theorem}

\begin{proposition} \label{54}
If $H$ is a pivotal Hopf algebra and $B$ is a copivotal Hopf algebra, then ${\mathcal{L}}^B_H$ is a pivotal category.
\end{proposition}

\begin{proof}
Suppose the pivot in $H$ is $\kappa$, the copivot on $B$ is $\varrho$. Define a $k$-linear map $\mathfrak{g}:B\rightarrow H$ by
\begin{eqnarray*}
\mathfrak{g}(b):= \varrho(b)\kappa,~~~~\mbox{for~any~}b \in B,
\end{eqnarray*}
it is easy to check that $\mathfrak{g}$ satisfies Eqs.(4.1)-(4.4). Since Theorem \ref{4.6}, the conclusion hold.
\end{proof}

\begin{example}
Let $G$ be a finite group and $H=kG$. Then ${\mathcal{L}}_H^H$ is a pivotal category, and the pivotal structure in ${\mathcal{L}}_H^H$ is totally determined by the center of $G$. Actually, assume that $\mathfrak{g}:H\rightarrow H$ is an entwined pivotal morphism of ${\mathcal{L}}_H^H$, then we have $\Delta(\mathfrak{g}(x)) = \mathfrak{g}(x) \o \mathfrak{g}(1_H)$ because Eq.(4.1), and $\Delta(\mathfrak{g}(x)) = \mathfrak{g}(x) \o \mathfrak{g}(x)$ since $H$ is a group algebra. Thus $\mathfrak{g}(x)=\mathfrak{g}(1_H)$ for all $x \in H$. From Eq.(4.3), one can get that $\mathfrak{g}(1_H)$ is in the center of $H$.
Eq.(4.4) is satisfied naturally.
\end{example}

\begin{example}\label{ldp}
Let $k$ be a field and $H_4$ be the Sweedler's 4-dimensional Hopf algebra in Example \ref{sw}.
Then recall from Theorem \ref{4.6}, it is easy to check that the entwined pivotal morphism $\mathfrak{g}\in hom_k(H_4,H_4)$ is defined as follows
\begin{eqnarray*}
\mathfrak{g}(1_H)=e, ~~\mathfrak{g}(e)=-e, ~~\mathfrak{g}(x)=\mathfrak{g}(y)=0.
\end{eqnarray*}
\end{example}

Note that if $(H,\mathrm{R})$ is a quasitriangular Hopf algebra and $(B,\beta)$ is a coquasitriangular Hopf algebra, then we have a braiding in ${\mathcal{L}}^B_H$:
\begin{equation*}
\begin{split}
\mathbf{C}_{M,N}: M\otimes N &\longrightarrow N\otimes M \\
m\otimes n &\longmapsto \sum \beta(m_{(1)},n_{(1)})n_{(0)}\cdot \mathrm{R}^{(2)}\otimes m_{(0)}\cdot \mathrm{R}^{(1)}.
\end{split}
\end{equation*}
Furthermore, the map $\mathbf{R} \in Hom_k(B \o B, H\o H)$ which make $(B,H,\dot{\varphi},\mathbf{R})$ into a double quantum group is defined as follows
\begin{equation*}
\begin{split}
\mathbf{R}: B\otimes B &\longrightarrow H \otimes H  \\
a \otimes  b &\longmapsto \sum  \beta(a,b) \mathrm{R}^{(2)}  \otimes \mathrm{R}^{(1)}.
\end{split}
\end{equation*}

\begin{theorem} \label{7.2}
Under the condition above, ${\mathcal{L}}^B_H$ is a ribbon category if and only if there is a
$k$-linear map $g:B\rightarrow H$, which is convolution invertible and satisfies the following identities for any $a,b \in B$, $h \in H$:
\\
$\left\{\begin{array}{l}
(LR1)~~g(a)h = h g(a);\\
(LR2)~~g(a_1) \o a_2 = g(a_2) \o a_1;\\
(LR3)~~\Delta(g(ab)) = \beta(a_3,b_3)\beta(b_2,a_2)g(a_1)\mathrm{R}^{(1)}\mathrm{R}'^{(2)} \o g(b_1)\mathrm{R}^{(2)}\mathrm{R}'^{(1)};\\
(LR4)~~g(a) = S^{-1}_Hg S_B(a),\mbox{~or~equivalently,~} S_H \ci g = g \ci S_B.
\end{array}\right.$
\\
\end{theorem}

\begin{proposition} \label{lonrib}
If $H$ is a ribbon Hopf algebra and $B$ is a coribbon Hopf algebra, then ${\mathcal{L}}^B_H$ is a ribbon category.
\end{proposition}

\begin{proof}
Suppose the ribbon element in $H$ is $\xi$, the ribbon character on $B$ is $\zeta$. Define a $k$-linear map $g:B\rightarrow H$ by
\begin{eqnarray*}
g(b):= \zeta(b)\xi, ~~~~\mbox{for~any~}b \in B,
\end{eqnarray*}
it is easy to check that $g$ satisfies Eqs.(5.1)-(5.4). Since Theorem \ref{5.9}, the conclusion hold.
\end{proof}

\vskip 0.5cm
\subsection{\textbf{Yetter-Drinfeld modules}}
\def\theequation{7.2.\arabic{equation}}
\setcounter{equation} {0} \hskip\parindent
\vskip 0.5cm

 Let $H$ be a finite dimensional Hopf algebra over $k$. Recall that if $M$ is both a right $H$-module and a right $B$-comodule, and satisfies
$$ \rho (m\cdot h) = \sum m_{(0)}\cdot a_2\otimes S(a_1)m_{(1)}a_3$$
for any $h\in H$, $m\in M$, then $M$ is a \emph{right-right Yetter-Drinfeld module}. The category of Yetter-Drinfeld modules and $H$-linear $H$-colinear homomorphisms is denoted by ${\mathcal{YD}}_H^H$.

If we define
\begin{equation*}
\begin{split}
\ddot{\varphi}: H\otimes H  &\longrightarrow H\otimes H \\
c\otimes a &\longmapsto a_{\ddot{\varphi}} \otimes c^{\ddot{\varphi}} := \sum a_2\otimes S(a_1) c a_3.
\end{split}
\end{equation*}
It is straightforward to show $\ddot{\varphi}$ is a right-right entwining structure, and $\mathcal {C}^H_H(\ddot{\varphi}) = \mathcal{YD}^H_H$.
Further,
it is obviously to see that the entwined smash product $H^{\ast op} \o H$ is the Drinfeld double of $H$, and the entwined smash coproduct of $(H,H,\ddot{\varphi})$ is the \emph{Drinfeld codouble} (see \cite{sm2}, Section 10) of $H$.

We can get the following results from Theorem \ref{4.6} and Theorem \ref{5.9}.

\begin{theorem} \label{ydpr}
\emph{(1)} $\mathcal{YD}^H_H$ is a pivotal category if and only if there is a
$k$-linear map $\mathfrak{g}:H\rightarrow H$, such that for any $a,b \in H$, the following conditions hold:
\\
$\left\{\begin{array}{l}
(YP1)~~\Delta(\mathfrak{g}(ab))=g(a) \o g(b);\\
(YP2)~~\varepsilon(\mathfrak{g}(1_H))=1_k;\\
(YP3)~~\mathfrak{g}(b)S^2(a) = a_2 \mathfrak{g}(S(a_1) b a_3 );\\
(YP4)~~\mathfrak{g}(a_1) \o S^2(a_2) =  \mathfrak{g}(a_2) \o S(g(1_H)) a_1 g(1_H).
\end{array}\right.$
\\

\emph{(2)} Since $\mathcal{YD}^H_H$ is a braided category with the braiding
\begin{eqnarray*}
t_{M,N}:M \otimes N\rightarrow N\otimes M,~~m\otimes n \mapsto n_{(0)}\otimes m \cdot n_{(1)}, \mbox{~where~}M,N \in \mathcal{YD}^H_H, ~m \in M,~n \in N,
\end{eqnarray*}
we immediately get that $(H,H,\ddot{\varphi},\mathbf{R})$ is a double quantum group, where $\mathbf{R}$ is defined by
\begin{equation*}
\begin{split}
\mathbf{R}: H\otimes H &\longrightarrow H \otimes H  \\
a \otimes  b &\longmapsto 1_H  \otimes \v(a) b.
\end{split}
\end{equation*}
Then $\mathcal{YD}^H_H$ is a ribbon category if and only if there is a
$k$-linear map $g:H\rightarrow H$, which is convolution invertible and satisfies the following identities for any $a,b\in H$:
\\
$\left\{\begin{array}{l}
(YR1)~~g(b) a = a_2 g( S(a_1) b a_3 );\\
(YR2)~~g(a_1) \o a_2 = {g(a_2)}_2 \o S({g(a_2)}_1)h_1({g(a_2)}_3);\\
(YR3)~~\Delta(g(ab)) = g(a_1)b_3 \o g(b_1)S(b_2)a_2b_4;\\
(YR4)~~g(b) = \sum {a_i}_2 a^i( S^{-1} g S ( S({a_i}_1) b {a_i}_3 ) ) ,
\end{array}\right.$
\\
where $a_i$ and $a^i$ are bases of $A$ and $A^\ast$ respectively, dual to each other.
\end{theorem}

Recall from \cite{Kassel} (or \cite{sm2}) that ${\mathcal{YD}}^H_H$ is equivalent to the center of the category of $Rep(H)$, we immediately get the following property.

\begin{proposition}
If $H$ is a pivotal Hopf algebra or a copivotal Hopf algebra, then ${\mathcal{YD}}^H_H$ is a pivotal category.
\end{proposition}

\begin{proof}
If $H$ is a pivotal Hopf algebra with the pivot $\kappa$, then we can define a $k$-linear map
$\mathfrak{g}:H\rightarrow H$ via $\mathfrak{g}(x) = \varepsilon(x)\kappa$. It is a direct computation to check that $\mathfrak{g}$ is an entwined pivotal morphism of ${\mathcal{YD}}^H_H$.

Similarly, if $H$ is a copivotal Hopf algebra with the pivotal character $\varrho$, then define
$\mathfrak{g}':H\rightarrow H$ by $\mathfrak{g}'(x) = \varrho(x)1_H$. It is easy to check that $\mathfrak{g}'$ also satisfies Eqs.(4.1)-(4.4).
\end{proof}

\begin{example}\label{ydp}
Consider the Sweedler's 4-dimensional Hopf algebra.
It is straightly to check that the entwined pivotal morphisms in ${\mathcal{YD}}^{H_4}_{H_4}$ are given by
\begin{eqnarray*}
\mathfrak{g}_1(1_H)=1_H, ~~\mathfrak{g}_1(e)=-1_H, ~~\mathfrak{g}_1(x)=\mathfrak{g}_1(y)=0;\\
\mathfrak{g}_2(1_H)=e, ~~\mathfrak{g}_2(e)=e, ~~\mathfrak{g}_2(x)=\mathfrak{g}_2(y)=0.
\end{eqnarray*}
\end{example}

\begin{proposition}
If $H$ is a ribbon Hopf algebra or a coribbon Hopf algebra, then ${\mathcal{YD}}^H_H$ is a ribbon category.
\end{proposition}

\begin{proof} If the ribbon element in $H$ is $\xi$, then we can define a $k$-linear map
$g:H\rightarrow H$ via $g(x) = \varepsilon(x)\xi$. It is a direct computation to check that $g$ is an entwined ribbon morphism of ${\mathcal{YD}}^H_H$.

Similarly, if the ribbon character of $H$ is $\zeta$, then we can define define
$g':H\rightarrow H$ by $g'(x) = \zeta(x)1_H$. It is easy to check that $g'$ also satisfies Eqs.(5.1)-(5.4).
\end{proof}

\

\begin{center}
 {\bf ACKNOWLEDGEMENT}
\end{center}
The work was partially supported by the NSF of China (NO. 11371088), and the NSF of Qufu Normal University (NO. xkj201514).
\\

 \end{document}